\documentclass{imsart}

\RequirePackage{amsthm,amsmath}
\RequirePackage[numbers]{natbib}
\RequirePackage[colorlinks,citecolor=blue,urlcolor=blue]{hyperref}
\RequirePackage{graphicx}

\usepackage{amsfonts,amssymb,float,mathrsfs,booktabs,color,epsfig,subfigure,url}
\usepackage{algorithm}
\usepackage{algorithmic}
\usepackage[title]{appendix}
\usepackage {multirow,bm,microtype} 
\usepackage[T1]{fontenc}
\usepackage[utf8]{inputenc}
\usepackage{babel}

\startlocaldefs
\theoremstyle{plain}
\newtheorem{theorem}{Theorem}[section]

\newtheorem{corollary}{Corollary}[section]

\theoremstyle{definition}
\newtheorem{example}{Example}
\newtheorem{remark}{Remark}
\newtheorem{assumption}{Assumption}
\endlocaldefs

\newcommand{\bx}{{\bf x}}
\newcommand{\um}{\underline{m}}
\newcommand{\bF}{{\bf F}}
\newcommand{\bbJ}{{\mathcal J}}
\newcommand{\bbS}{{\bf S}}
\newcommand{\bY}{{\bf Y}}
\newcommand{\by}{{\bf y}}
\newcommand{\bX}{{\bf X}}
\newcommand{\gS}{{\bm \Sigma}}
\newcommand{\rE}{{\rm E}}
\newcommand{\md}{{\rm d}}
\newcommand{\non}{\nonumber\\}

\begin{document}

\begin{frontmatter}
\title{Inference on  testing  the number of spikes in a   high-dimensional generalized  spiked Fisher matrix}

\runtitle{Inference on testing  the number of spikes}

\begin{aug}
\author{\fnms{Rui} \snm{ Wang}\thanksref{1}\ead[label=e1]{wangrui\_math@stu.xjtu.edu.cn}}
\and
\author{\fnms{Dandan} \snm{Jiang}\thanksref{1,t1}\ead[label=e2]{jiangdd@mail.xjtu.edu.cn}}

\address[1]{School of Mathematics and Statistics, Xi'an Jiaotong University, Xi'an, China\\
\printead{e1,e2}}

%\author{\fnms{Third} \snm{Author}
%\ead[label=e3]{third@somewhere.com}
%\ead[label=u1,url]{www.foo.com}}

%\address{Address of the Third author\\
%usually few lines long\\
%usually few lines long\\
%\printead{e3}\\
%\printead{u1}}

\thankstext{t1}{Corresponding author: Dandan Jiang.}
\runauthor{R. Wang and D. Jiang}

\end{aug}

\begin{abstract}
	The spiked Fisher matrix is a significant topic for two-sample problems in multivariate statistical inference. This paper is dedicated to testing the number of spikes in a high-dimensional generalized spiked Fisher matrix that relaxes the Gaussian population assumption and the diagonal constraints on the population covariance matrices.
	First, we propose a general test statistic predicated on partial linear spectral statistics to test the number of spikes, then establish the central limit theorem (CLT) for this statistic under the null hypothesis. Second, we apply the CLT to address two  statistical problems: 
	variable selection in high-dimensional linear regression and change point detection. 
	For each test problem, we construct new statistics and derive their asymptotic distributions under the null hypothesis. Finally, simulations and empirical analysis are conducted to demonstrate the remarkable effectiveness and   generality of our proposed methods  across various scenarios.
\end{abstract}


\tableofcontents
\end{frontmatter}

\section{Introduction}

The spiked model, initially introduced by \cite{Johnstone2001},  describes  that some extreme eigenvalues of random matrices exhibit significant separation from the remainder. This model plays a pivotal role in multivariate statistical inference and has a wide range of applications in many modern fields, such as wireless communication \cite{Telatar1999, TOSE2022}, speech recognition \cite{Hastie1995, Johnstone2001}, and other scientific fields.
A critical aspect of  the spiked model is to determine the number of spiked-eigenvalues (spikes), which helps reveal the latent dimensionality of data and the underlying structure of the population. 
Notably, due to the close relationship between the spiked model and  principal component analysis (PCA), and factor analysis (FA), determining the number of principal components (PCs) or factors is  equivalent to finding the number of spikes. 

Numerous studies have been devoted to this topic in one-sample spiked population models.   One prominent type of research focuses on information criteria, such as \cite{BN2002}, which developed a constructive work to  provide several information criteria for estimating the number of factors. Following their work, \cite{Forni2009, Aless2010, Pass2017} et al. improved these criteria using different tools. 
Another main type of research relies on the asymptotic behavior of eigenvalues. The reference \cite{KN2008} estimated the number of PCs  by comparing the sample eigenvalues  against a threshold derived from the Tracy-Widom Law. 
Alternatively, \cite{PY2012} and \cite{PY2014} developed approaches based on differences in consecutive sample eigenvalues.
More recently, \cite{Jiang2022} proposed a partial linear spectral statistic to test the number of spikes in generalized spiked covariance matrices. Despite these advancements, the studies on determining the number of spikes for two-sample spiked models  has  received less attention.

The two-sample spiked model addresses the problems involving two  $p$-dimensional populations with covariance matrices $\mathbf \Sigma_1$ and $\mathbf \Sigma_2$. 
When the matrix $\gS_1\gS_2^{-1}$ possesses a few extreme eigenvalues that are significantly distinct from the remaining ones,
the Fisher matrix $\mathbf S_1\mathbf S_2^{-1}$ is called the spiked Fisher matrix, where $\mathbf S_1$ and $\mathbf S_2$ are the corresponding sample covariance matrices. 
Determining the number of spikes in such a two-sample spiked model was  initially studied in  \cite{wang2017}
with a simple formation of
\begin{align}\label{sigma1}
	\gS_1=\mathbf \gS_2+\mathbf\Delta,
\end{align} 
where  $\mathbf\Delta$ is a symmetric matrix with low rank. In their study, \cite{wang2017}  compared the sample eigenvalues against a threshold in a similar way to \cite{KN2008}.
However, when $\gS_1\gS_2^{-1}$ satisfies the condition in \eqref{sigma1}, it inherently has equal non-spiked eigenvalues of 1, thereby restricting its applications to more complex covariance structure. Subsequently, motivated by the hypothesis testing problem $\mathcal{H}_{0}: \mathbf\Sigma_1= \sigma^2\mathbf\Sigma_2 \ \text{v.s.} \ \mathcal{H}_{1}:  \mathbf\Sigma_1=\sigma^2\mathbf\Sigma_2+\mathbf\Delta$, \cite{zeng2022}  explored a more general spiked Fisher matrix, where the  non-spiked eigenvalues are still assumed to be identical. They introduced a generic criterion to estimate the number of spikes, later extended by \cite{zeng2023} to scenarios where the number of spikes diverges as $p$ goes to infinity. 
Nevertheless, all of these models are unsuitable for more complex structures characterized by varying non-spiked eigenvalues.
 This leads to the  study of the so-called generalized  spiked Fisher matrix, as investigated in \cite{jiang2021} and \cite{Jiang2023}, where $\gS_1\gS_2^{-1}$  has the spectral form  
\begin{equation} \label{spec}
	\underbrace{\alpha_1, \cdots, \alpha_1}_{m_1}, \ldots, \underbrace{\alpha_K, \cdots, \alpha_K}_{m_K}, \beta_{p, M+1}, \cdots, \beta_{p, p}.
\end{equation}
Here, the spectrum in \eqref{spec} is arranged in descending order,  where $\alpha_1, \alpha_2,\cdots, \alpha_K$ are  spikes with multiplicity $ m_1,m_2,\cdots, m_K$, satisfying $m_1+m_2+\cdots+ m_K= M$ with $M$ being a fixed integer. These spikes are significantly larger than the remaining non-spiked eigenvalues $\beta_{p, M+1}, \cdots, \beta_{p, p}$.

In this paper, we focus on testing the number of spikes in a generalized spiked Fisher matrix, as described in  \eqref{spec}, when both the sample size and the dimension of variables proportionally tend to infinity. We first propose a general test statistic that relies on a partial linear spectral statistic, and establish its asymptotic normality under the null hypothesis.  This test is then applied to two statistical problems: testing the number of significant variables in a large-dimensional linear regression model and testing the change points in the sequence data.
Our approach offers several notable advantages compared to existing works.
Firstly, our method exhibits superior generality as it is free of population distributions of samples, and does not rely on the diagonal assumption of the population covariance matrices. 
Furthermore, the proposed method is more comprehensive with general assumptions of $\mathbf \Sigma_1$ and $\mathbf \Sigma_2$, and encompasses the low-rank case of $\mathbf \Sigma_1-\mathbf \Sigma_2$ as a special instance.
Finally, beyond theoretical advancements, our practical applications demonstrate commendable performance. In the variable selection problem for large-dimensional linear regression, our test accurately identifies the number of significant variables, exhibiting correct sizes and robust behavior. Additionally, in the change point detection problem,  our method consistently achieves higher accuracy across various scenarios compared to existing approaches.

The arrangement of this article is as follows. Section \ref{sec2} introduces the generalized spiked Fisher matrix, proposes a test for the number of spikes, and derives its asymptotic normality under the null hypothesis. Following this, Section \ref{sec3} and Section \ref{sec4} elaborate on two practical applications of the proposed test, the identification of the number of significant variables in the large-dimensional linear regression model and the  change point detection. Section \ref{sec5} conducts numerical and empirical studies for the above results. Finally, Section \ref{sec6} offers concluding remarks. Proofs are provided in the appendix.

\section{Testing the  number of spikes for a generalized spiked Fisher matrix}\label{sec2} 
\subsection{Model formulation }
Following \cite{jiang2021},  the generalized spiked Fisher matrix 
serves as the primary target matrix for the statistical problems within the two-sample spiked model, which is formulated as follows. 
Let $\gS_1$ and $\gS_2$ be two general independent and deterministic  matrices, where $\gS_1$ is non-negative definite, and $\gS_2$ is positive definite.  To begin with, we define 
\begin{align*}
	\bX&=(\bx_1,\bx_2,\cdots,\bx_{n_1})=(x_{ij}, 1\leq i\leq p,1 \leq j\leq n_1),\\
	\bY&=(\by_1,\by_2,\cdots,\by_{n_2})=(y_{il}, 1\leq i\leq p,1 \leq l\leq n_2),
\end{align*}
where $\{x_{ij},i,j = 1,2,...\} $ and  $\{y_{il},i,l = 1,2,...\} $ are either both real or both complex random variable arrays,  satisfying the following assumption.
\begin{assumption} \label{assum1}
	The two independent double arrays $\{x_{ij},i,j = 1,2,...\} $
	and  $\{y_{il},i,l = 1,2,...\} $ are independent and identically distributed (i.i.d.) random variables with mean 0 and variance 1. 
	$\rE (x_{ij}^2)=\rE (y_{ij}^2)=0$ for the complex case. Moreover, $\rE|x_{ij}|^4 <\infty$ and $\rE|y_{ij}|^4 <\infty$.
\end{assumption} 	
Given this assumption, $\gS_1^{1/2}\bX$ and $\gS_2^{1/2}\bY$ can be seen as two independent random samples from the $p$-dimensional populations  with general population  covariance matrices $\gS_1$ and $\gS_2$, respectively.
The  corresponding  sample covariance matrices of the two observations are  given by
\begin{equation}
	{\bbS}_1=\frac{1}{n_1}\gS_1^{\frac{1}{2}}
	{\bX}{\bX}^*{\gS}_1^{\frac{1}{2}}  \quad \text{and} \quad   
	{\bbS}_2=\frac{1}{n_2}{\gS}_2^{\frac{1}{2}}
	{\bY}{\bY}^*{\gS}_2^{\frac{1}{2}}.\label{S12}
\end{equation}
Then,  the Fisher matrix is defined as
\begin{equation}
	\bF=\bbS_1\bbS_2^{-1}, \label{F_def}
\end{equation}
which shares the same non-zero eigenvalues as the matrix 
\begin{equation}
	{\mathbf F}={\mathbf R}_p^{*}\widetilde {\mathbf S}_1{\mathbf R}_p\widetilde {\mathbf S}_2^{-1}.\label{F}
\end{equation}
Here, ${\mathbf R}_p={\boldsymbol\Sigma}_1^{1/2}{\boldsymbol\Sigma}_2^{-{1/2}}$,   $\widetilde {\mathbf S}_1=n_1^{-1}{\mathbf X}{\mathbf X}^*$ and $\widetilde {\mathbf S}_2={n_2}^{-1}{\mathbf Y}{\mathbf Y}^*$  are the standardized sample covariance matrices. Thus, if there is no confusion, we denote the matrix  in (\ref{F}) as the Fisher matrix $\bF$ and study it from this point onward.

The eigenvalues of ${\mathbf R}_p^*{\mathbf R}_p={\boldsymbol\Sigma}_2^{-{1/2}}{\boldsymbol\Sigma}_1{\boldsymbol\Sigma}_2^{-{1/2}}$ are equivalent to those in the matrix ${\boldsymbol\Sigma}_1{\boldsymbol\Sigma}_2^{-1}$. When the difference between $\boldsymbol\Sigma_1$ and $\boldsymbol\Sigma_2$ is general and does not adhere to the low-rank  assumption, the spectrum of ${\mathbf R}_p^*{\mathbf R}_p$ aligns with the form specified in \eqref{spec} and can be rewritten in  descending order as
$$
	\beta_{p, 1}, \beta_{p, 2}, \cdots, \beta_{p, p}.
$$
Let  $\beta_{p,j}=\alpha_k, j \in  \bbJ_k$, where  $\bbJ_k$ denotes the set consisting of the ranks of the $ m_k$-ple eigenvalue  $\alpha_k$. Then, $\alpha_1,\alpha_2, \cdots, \alpha_K$ are  the population  spiked eigenvalues with respective multiplicity $ m_k, k=1,2,\cdots,K$,   satisfying $ m_1+m_2+\cdots+ m_K= M$, where $M$ is a fixed integer. 
%Moreover, the largest spike is allowed to tend to infinity. 
The rest $\beta_{p, M+1}, \ldots, \beta_{p, p} $ are non-spiked eigenvalues.
Then, we call  $\bF$ defined in (\ref{F})  as the so-called generalized spiked Fisher matrix by assuming that these spikes are much larger than the remaining eigenvalues. 

Throughout the paper, the empirical spectral distribution (ESD) of a  $p\times p$ matrix $\mathbf A$, with eigenvalues denoted by $\{l_i(\mathbf A)\}$, is defined as the probability measure $p^{-1}\sum_{i=1}^{p}\delta_{\{l_i(\mathbf A)\}}$, where  $\delta_{\{a\}}$ denotes the Dirac mass at a point $a$.  For a given sequence of	random matrices $\{\mathbf A_n\}$, the  ESD  of $\{\mathbf A_n\}$ converges, in a sense to be precise and when $n\rightarrow \infty$, to a limiting spectral distribution (LSD). To proceed further, we consider some regular assumptions in  random matrix theory.
\begin{assumption} \label{assum2}
	 Assuming that $c_{n_1}=p/n_1\rightarrow c_1\in (0, \infty)\quad \text{and} \quad  c_{n_2} =p/n_2 \rightarrow c_2\in (0,1)$ %\label{limitS}$
as  $\min( p,n_1,n_2)\rightarrow \infty$.
\end{assumption}

\begin{assumption} \label{assum3}
	The matrix  ${\mathbf R}_p$ is a deterministic and  non-negative
	definite Hermitian matrix, with all its eigenvalues bounded. 
	 Moreover, the  ESD $ H_n(t)$ of   ${\mathbf R}_p^*{\mathbf R}_p$,  generated by the population eigenvalues $\{\beta_{p,j}, j=1,2,\cdots,p\}$,  tends to a proper probability measure $ H(t)$  as $\min( p,n_1,n_2)\rightarrow \infty$. 
\end{assumption}

For the generalized spiked Fisher matrix in \eqref{F}, let its sample eigenvalues be denoted as $l_{1}(\bF) \geq l_{2}(\bF)\geq\cdots\geq l_{p}(\bF)$. To apply random matrix theory to hypothesis testing problems, we need to establish the central limit theory (CLT) of linear spectral statistics composed of these eigenvalues. To this end, we first focus on their first-order limits. Particularly, for the  sample spiked eigenvalues $\{l_{j}(\bF), j \in \bbJ_k\}$, $k=1,2,\ldots,K$ associated with the population spiked eigenvalue $\alpha_k$, \cite{jiang2021} obtained  the following relationship  under the separation condition $\min_{i\neq k}|\alpha_k/ \alpha_i -1| > d$ for some constant $d>0$:
\begin{equation}
	\frac{l_{j}(\bF)}{\psi_k}-1 \to 0, a.s.\quad ~  j \in  \bbJ_k, 
	\label{lpsi1}
\end{equation}
 where $\psi_k$ is a function with respect to $\alpha_k$, specifically expressed as
\begin{align}\label{psik}
	\psi_k:=\psi( \alpha_k)=\frac{\alpha_k\left(1-c_1\int \displaystyle\frac{t}{t-\alpha_k}{\rm d}H(t)\right)}{1+c_2\int\displaystyle\frac{\alpha_k}{t-\alpha_k}{\rm d}H(t)}.
\end{align}

\subsection{Test on the number of spikes: the general case}
To test  the number of  spikes in a generalized   spiked  Fisher matrix, we focus on the following hypothesis 
\begin{equation}
	\mathcal{H}_{0}:  M=  M_0 \quad \mbox{v.s.} \quad \mathcal{H}_{1}: M \neq   M_0,
	\label{H1}
\end{equation}
where $M_0$ is the true number of  spikes in ${\mathbf R}_p^*{\mathbf R}_p$.

Motivated by a classic likelihood ratio test statistic for testing the number of spikes in a one-sample spiked covariance matrix, we are going to propose a similar statistic for the generalized spiked Fisher matrix.
More specifically,  it is well-known that the  likelihood ratio  test statistic $-2\log L$ is often used for the probability PCA  model, where
\begin{align*}
	L=\left( \frac{(\prod_{i=M+1}^{p}l_i)^{\frac{1}{p-M}}}{\sum_{i=M+1}^{p}l_i/(p-M)}\right)^{\frac{(p-M)n}{2}}. 
\end{align*}	
It can be found that $-2\log L$ primarily depends on $\sum_{i=M+1}^{p}l_i$ and  $\sum_{i=M+1}^{p}\log(l_i)$.
Inspired by this, for the hypothesis \eqref{H1} in the  generalized spiked  Fisher matrix,  we proposed the  test statistic
\begin{align}\label{newtest2}
	\sum_{j=1}^p  f\{l_j(\bF)\}-\sum_{j \in  \bbJ_k, k=1}^K  f\{ l_j(\bF)\}.
\end{align}
Here, $f \in  \mathcal{ A}$,  and $\mathcal{ A}$  represents a set of analytic functions defined on an open set of the complex plane containing  the supporting set of the LSD of the matrix $\mathbf F$. In statistical applications, the form of $f(x)$ varies from case to case, depending on the particular problem under consideration.

There has been some related work for the limit behavior of the sample eigenvalues of $\mathbf F$.
For the linear spectral statistic $\sum_{j=1}^p  f\{l_j(\bF)\}$, \cite{zheng2017} developed its CLT  applicable when the general Fisher matrix has no spikes, while \cite{jiang2021} focused on the limit of sample spiked eigenvalues  $l_j(\bF), j \in \bbJ_k, k=1,2,\ldots, K$.
However, neither \cite{zheng2017} nor \cite{jiang2021} can address the problem of testing the number of spikes individually. 
Therefore, we aim to establish the CLT of the partial linear spectral statistic in \eqref{newtest2} under the null, offering a fresh perspective on this statistical problem.

To develop the asymptotic theory of the  statistic in \eqref{newtest2}, we  first introduce some intermediate notations. Let $F^{(c_1,c_2,H)}$ represent the LSD of  the Fisher matrix $\bF$, and $F^{(c_{n1},c_{n2},H_n)}$  be the counterpart of $F^{(c_1,c_2,H)}$  with the substitution of $ (c_1,c_2, H)$ with  $ (c_{n1},c_{n2}, H_n)$.
Moreover, we denote $m_{c_1,c_2}(z)$ as the Stieltjes transform of $F^{(c_1,c_2,H)}$, and $\um_{c_1,c_2}(z) =-(1-c_1)z^{-1}+c_1m_{c_1,c_2}(z)$. For brevity, the notation $\um_{c_1,c_2}(z)$ will be simplified to  $\um(z)$.
Additionally, let $\um_{c_2}(z)$ satisfy  the equation 
$z=-\{\um_{c_2}(z)\}^{-1}+c_2 \int \{t+\um_{c_2}(z)\}^{-1} \md  H (t)$.
We further define  $h^2=c_1+c_2-c_1c_2$ and $m_0(z)=\um_{c_2}\{-\um(z)\}$. 
Then, the CLT  of the  statistic in \eqref{newtest2} under the null hypothesis  is established as follows,  with its proof provided in Appendix \ref{app1}. 
\begin{theorem}\label{Th1}
	For the hypothesis testing problem \eqref{H1},    suppose that  Assumptions~\ref{assum1} to \ref{assum3} are satisfied. Then, under the null,   the 
	test statistic \eqref{newtest2} follows 
	\begin{align}
		T_{f,H}&=\nu_{ f,H}^{-\frac{1}{2}}\left\{\sum\limits_{j=1}^p  f\{l_j(\bF)\}\!-\!\!\sum\limits_{j \in  \bbJ_k, k=1}^K  f\{ l_j(\bF)\}
		\!-\!d_1( f,  H_n)\!+\!d_2( f,\alpha,  H)\!-\! \mu_{ f,H}\right\} \notag\\
		&\Rightarrow \mathcal{N} \left( 0,
		1\right), \label{asymp1}
	\end{align}
	where 
\begin{align}
	&d_1( f,  H_n)=p\int  f(x)  {\rm d}  F^{(c_{n1},c_{n2},H_n)}(x),\label{d1}\\
	&d_2( f,\alpha_k,  H)=\sum\limits_{k=1}^K m_k  f\left(\psi_k\right).\notag
\end{align}
The mean and variance terms, $\mu_{ f,H}$ and $\nu_{f,H}$, are functions of $m_0$ and their expressions can be found  in the equations \eqref{mean} and \eqref{var}.
\end{theorem}

In Theorem \ref{Th1}, a notable distinction from the CLT in \cite{zheng2017} is the computation of the centering parameter $d_1(f,H_n)$ in \eqref{d1}, whose expression becomes more intricate for spiked Fisher matrices.
When the function $f(x)$ involved in test statistics and the population ESD $H_n(t)$ have simple forms, explicit expressions for the center term can  be computed, as demonstrated in the following section. 
However,  when $f(x)$ and $H_n(t)$ become complex, the analytical expressions are not readily available.
Similar to Remark 4.1 in \cite{zheng2017}, we propose a numerical procedure in the supplementary material to approximate  $d_1(f,H_n)$
for generalized spiked Fisher matrices, thereby facilitating practical implementation and enhancing the applicability of our approach.

\subsection{The low-rank case of $\mathbf \Sigma_1-\mathbf \Sigma_2$}
Recall the  spiked Fisher matrix adhering to the formula \eqref{sigma1}, the difference $\mathbf \Delta$ between $\boldsymbol\Sigma_1$ and $\boldsymbol\Sigma_2$ is  of low rank.
Thus, the spectrum of ${\mathbf R}_p^*{\mathbf R}_p$  is arranged  in  descending order as
$$
\beta_{p, 1},\beta_{p, 2}, \cdots,\beta_{p, M},1 \cdots, 1.
$$
Here,  $\beta_{p,j}=\alpha_k, j \in  \bbJ_k$,  where $\alpha_1,\alpha_2, \cdots, \alpha_K$ are  spikes with multiplicity $ m_1,m_2,\cdots, m_K$,  satisfying $ m_1+m_2+\cdots+ m_K= M$.  Within this framework, the LSD of $\mathbf R^*_p\mathbf R_p$  is $H(t)=\delta_{\{1\}}$.
Therefore, the CLT in Theorem \ref{Th1} can be  rewritten as the following result, where the mean and variance terms  take the form of contour integrals along the unit circle,  as demonstrated in \cite{zheng2012}.
\begin{corollary}\label{coro1}
	For the testing problem \eqref{H1},    suppose that  Assumptions~\ref{assum1} to \ref{assum3} and the condition $\mathbf\Sigma_1=\mathbf\Sigma_2+\mathbf \Delta$ are satisfied,  
	then the asymptotic  distribution  of the 
	test statistic \eqref{newtest2} can be expressed as
	\begin{align} \label{asymp2}
		T_{f,1}&=\nu_{ f,1}^{-\frac{1}{2}}\left\{\sum\limits_{j=1}^p  f\{l_j(\bF)\}\!-\!\!\sum\limits_{j \in  \bbJ_k, k=1}^K  f\{ l_j(\bF)\}
		\!-\!d_1( f,  H_n)\!-\! \mu_{ f,1}\right\} \Rightarrow \mathcal{N} \left( 0,
		1\right),
	\end{align}
	where 
	\begin{align*}
		&d_1( f,  H_n)=p\int  f(x)  {\rm d}  F^{(c_{n1},c_{n2},H_n)}(x)-\sum\limits_{k=1}^K m_k  f\left(\frac{\alpha_k(1-\alpha_k-c_1)}{1-\alpha_k+c_2\alpha_k}\right), \\
		&\mu_{f,1}=\lim _{r \downarrow 1}\! \frac{q}{4 \pi \mathrm{i}} \!\oint_{|\xi|=1}\! f\left(\!\frac{1+h^2+2 h \Re(\xi)}{\left(1-c_2\right)^2}\!\right)\! \left[\frac{1}{\xi-r^{-1}}\!+\!\frac{1}{\xi+r^{-1}}\!-\!\frac{2}{\xi+c_2/h}\right]\!\mathrm{d} \xi \notag \\
		&\qquad\quad+\frac{\beta_x  c_1\left(1-c_2\right)^2}{2 \pi \mathrm{i} h^2} \oint_{|\xi|=1} f\left(\frac{1+h^2+2 h  \Re(\xi)}{\left(1-c_2\right)^2}\right)\frac{1}{\left(\xi+c_2 / h\right)^3} \mathrm{~d} \xi \notag \\
		&\qquad\quad+\frac{\beta_y \left(1-c_2\right)}{4 \pi \mathrm{i}} \oint_{|\xi|=1} f\left(\frac{1+h^2+2 h \Re(\xi)}{\left(1-c_2\right)^2}\right)\frac{\xi^2-c_2 / h^2}{\left(\xi+c_2 / h\right)^2} \notag\\
		&\qquad\quad\quad \cdot
		\left[\frac{1}{\xi-\sqrt{c_2} / h}+\frac{1}{\xi+\sqrt{c_2} / h}-\frac{2}{c_2 / h}\right] \mathrm{d} \xi,
	\end{align*}
	and
	\begin{align*}
		\nu_{f,1}
		=&-\lim _{r \downarrow 1} \frac{q+1}{4 \pi^2} \oint_{\left|\xi_1\right|=1}\! \oint_{\left|\xi_2\right|=1}\!f\left(\!\frac{1+h^2+2 h \Re\left(\xi_1\!\right) }{\left(1-c_2\right)^2} \right) \!f\left(\!\frac{1+h^2+2 h \Re\left(\xi_2\right) }{\left(1-c_2\right)^2 }\! \right) \notag\\
		& \quad\cdot \frac{1}{\left(\xi_1-r \xi_2\right)^2}
		\mathrm{~d} \xi_1 \mathrm{~d} \xi_2 \notag \\
		&-\frac{\left(\beta_x c_1+\beta_y c_2\right)\left(1-c_2\right)^2}{4 \pi^2 h^2}\! \oint_{\left|\xi_1\right|=1} \!f\left(\!\frac{1+h^2+2 h \Re\left(\xi_1\right) }{\left(1-c_2\right)^2}\! \right) \!\frac{1}{\left(\xi_1\!+\!c_2 / h\right)^2}\! \mathrm{~d} \xi_1 \notag \\
		& \quad \cdot \oint_{\left|\xi_2\right|=1} f\left(\frac{1+h^2+2 h \Re\left(\xi_2\right) }{\left(1-c_2\right)^2} \right) \frac{1}{\left(\xi_2+c_2 / h\right)^2} \mathrm{~d} \xi_2.
	\end{align*}
	Here, $\Re\left(z\right)$ represents the real part of $z$, $\beta_x={\rE|x_{11}|^4-q-2}$ and $\beta_y={\rE|y_{11}|^4-q-2}$ with $q=1$ for real case and $q=0$ for complex. 
\end{corollary}

As  mentioned in the likelihood ratio statistics $-2\log L$, we  select functions $f(x) = x$ and $f(x) = \log x$ as two examples to  calculate   more concise results for Corollary \ref{coro1}   when $H(t) =\delta_{\{1\}}$.
The  calculation details are presented in the supplementary material.
\begin{example} \label{ex1}
	If $f(x)=x$ and $H(t)=\delta_{\{1\}}$, the statistic in (\ref{asymp1}) simplifies to
	$$
	T_{x, 1 }=\nu_{x, 1}^{-\frac{1}{2}}\left\{\sum_{j=1}^{p} l_{j}-\sum_{j \in \mathcal{J}_{k}, k=1}^{K} l_{j}-d_{x, 1}-\mu_{x, 1}\right\} \Rightarrow \mathcal{N}(0,1),
	$$
	where
	$$
	\begin{aligned}
		&d_{x, 1}=\frac{p-M+\sum_{k=1}^{K} m_{k} \alpha_{k}}{1-c_2}-\sum_{k=1}^{K}\frac{m_k\alpha_k(1-\alpha_k-c_1)}{1-\alpha_k+c_2 \alpha_k},\\
		&\mu_{x, 1}=\frac{qc_2}{(1-c_2)^2}+\frac{\beta_{y}c_2}{1-c_2},\quad \nu_{x, 1}=\frac{(q+1)h^2}{(1-c_2)^4}+\frac{\beta_{x}c_1+\beta_{y}c_2}{(1-c_2)^2}.
	\end{aligned}
	$$
\end{example}

\begin{example} \label{ex2}
	If $f(x)=\log x$ and $H(t)=\delta_{\{1\}}$, the statistic in (\ref{asymp1}) simplifies to
	$$
	T_{\log, 1}=\nu_{\log, 1}^{-\frac{1}{2}}\left\{\sum_{j=1}^{p} \log \left(l_{j}\right)-\sum_{j \in \mathcal{J}_{k}, k=1}^{K} \log \left(l_{j}\right)-d_{\log, 1}-\mu_{\log, 1}\right\} \Rightarrow \mathcal{N}(0,1),
	$$
	where
	\begin{align}
		d_{\log,1}=&p\left\{\frac{(1-c_2)\log(1-c_2)}{c_2}+\frac{(1-h^2)\left[\log(1-h^2)-\log(1-c_2)\right]}{c_2-h^2}
		\right\} \notag\\
		&-\sum_{k=1}^{K} m_{k} \log \left(\frac{1-\alpha_k-c_1}{1-\alpha_k+c_2 \alpha_k}\right), \notag
	\end{align}
	\begin{align}
		\mu_{\log, 1}=&\frac{q}{2} \log\frac{1-h^2}{(1-c_2)^2}-\frac{\beta_{x}c_1-\beta_{y}c_2}{2}, \notag\\
	 \nu_{\log, 1}=&-(q+1) \log (1-h^2)+(\beta_{x}c_1+\beta_{y}c_2). \notag
	\end{align}

\end{example}

As presented in Corollary \ref{coro1}, the  mean and variance terms in Examples \ref{ex1} and \ref{ex2} can  be computed using the residues. When $f(x) = x$ and $f(x) = \log x$, some  results relevant to computing contour integrals  can be found in the references \cite{Jiang2017} and \cite {zheng2012}, which are utilized in the proofs of the above examples.

\section{Application  to   large-dimensional linear regression models}\label{sec3}
In large-dimensional linear regression models, identifying the number of significant variables is crucial for helping measure the complexity of the model, which assists in further refining selection algorithms, such as LASSO or step-wise regression.  This approach helps in focusing on truly influential predictors, which is useful in modern applications with large datasets.

In testing for significant variables, statistical tools like Wilk's likelihood ratio test (LRT) play an essential role. 
In large-dimensional settings, \cite{Bai2013} developed a modified Wilk’s test for the general linear hypothesis using random matrix theory, with the test statistic expressed as a function of eigenvalues of the standard Fisher matrix. However, their method cannot establish a direct relation with the number of significant variables and is thus unsuitable for tackling our current issue. To address it, we introduce a novel test statistic predicated on the eigenvalues of the spiked Fisher matrix, where the number of spikes in this matrix is equivalent to the number of significant variables. This approach offers a more targeted methodology for variable selection in high-dimensional data analysis.

Let the $p$-dimensional observations $\left\{\mathbf{z}_{i}\right\}_{1\leqslant i \leqslant n}$  being from the following linear regression model: 
\begin{align}\label{linear}
	\mathbf{z}_{i}=\mathbf{B} \mathbf{w}_{i}+\boldsymbol\epsilon_{i}, \quad i=1,2, \ldots, n,
\end{align}
where $\{\mathbf{w}_{i}\}_{1\leqslant i\leqslant n}$ are  known design vectors (or regression variables) of dimension $r$, $\mathbf{B}$ is a $p \times r$ matrix of regression coefficients, and $\{\boldsymbol\epsilon_{i}\}_{1\leqslant i\leqslant n}$ are a sequence of i.i.d. errors drawn from $\mathcal{N}_{p}(\mathbf 0, \mathbf{V})$. We  assume that $n \geqslant p+r$ and  the rank of $\mathbf{W}=\left(\mathbf{w}_{1}, \mathbf{w}_{2},\cdots, \mathbf{w}_{n}\right)$ is $r$.

As a common manipulation, we partition the matrix  $\mathbf{B}$ as $\left(\mathbf{B}_{1}, \mathbf{B}_{2}\right)$, where  $\mathbf{B}_{1}$ has $r_{1}$ columns and $\mathbf{B}_{2}$ has $r_{2}$ columns, satisfying $r_{1}+r_{2}=r$. 
To deal with large-dimensional data, we assume that $p/r_1 \rightarrow c_{1} \in(0,\infty)$, $ p/(n-r) \rightarrow c_{2} \in(0,1)$ as  $\min(p,r_1,n-r) \rightarrow \infty$. Assuming that $M_0$ is the true number of non-zero columns of  $\mathbf{B}_{1}$, which is the true number of   significant variables. Our goal is to test the hypothesis that:
\begin{align} \label{H2}
	\mathcal H_{0}: M=M_0 \quad v.s. \quad \mathcal H_{1}: M \neq M_0.
\end{align}

Divide $\mathbf{w_i}^T$  into $\left(\mathbf{w}_{i, 1}^{T}, \mathbf{w}_{i, 2}^{T}\right)$  in a same manner corresponding to the partition of $\mathbf B=\left(\mathbf{B}_{1}, \mathbf{B}_{2}\right)$.  Then, as elaborated in \cite{Bai2013},  the  well-known  Wilk’s statistic  can be written in terms of Wishart-distributed matrices as
\begin{align} \label{Lambda}
	{\Lambda}_{n}=\left|\mathbf{I}+\frac{r_{1}}{n-r}\mathbf H_0\mathbf G^{-1} \right|^{-1},
\end{align}
where $ \mathbf H_0=(\hat{\mathbf{B}}_{1}-\mathbf B_1) \mathbf{A}_{11: 2}(\hat{\mathbf{B}}_{1}-\mathbf B_1)^{T}/r_1, \mathbf G=n\hat{\mathbf{V}}/(n-r)$. 
Here,  $\mathbf{A}_{11: 2}=\sum_{i=1}^n \mathbf{w}_{i,1} \mathbf{w}_{i,1}^{T}-\sum_{i=1}^n \mathbf{w}_{i,1} \mathbf{w}_{i,2}^{T}\left(\sum_{i=1}^n \mathbf{w}_{i,2} \mathbf{w}_{i,2}^{T}\right)^{-1} \sum_{i=1}^n \mathbf{w}_{i,2} \mathbf{w}_{i,1}^{T}$,
and $\hat{\mathbf{B}}_{1}$ consists of the first $r_{1}$ columns  of $\hat{\mathbf{B}}$ with $\hat{\mathbf{B}}$ being the maximum likelihood estimator of $\mathbf B$ in the full parameter space.  Under the null, by Lemmas 8.4.1 and 8.4.2 in \cite{Anderson2003}, we can obtain that
\begin{align} \label{wishart}
	(n-r)\mathbf G \sim W_{p}(\mathbf{V}, n-r),\quad r_1\mathbf H_0 \sim W_{p}(\mathbf{V}, r_1),
\end{align}
and  $\mathbf H_0$ is independent of $\mathbf{G} $, where $W_{p}(\cdot,\kappa)$ represents the $p$-dimensional Wishart distribution with $\kappa$ degrees of freedom. Therefore, $\mathbf H_0\mathbf G^{-1}$ is  distributed as the Fisher matrix.
According to \cite{Anderson2003}, the logarithm-transformed LRT, $-\log {\Lambda}_{n}$,   is traditionally  employed as a test statistic for the general linear hypothesis.

To apply Theorem \ref{Th1} to test the hypothesis \eqref{H2}, we first construct a spiked Fisher matrix with $M_0$ spikes. We start by decomposing $\mathbf{H}_0$ as $\mathbf{H}_0=(\hat{\mathbf{B}}_{1} \mathbf{A}_{11: 2}\hat{\mathbf{B}}_{1}^{T}-\hat{\mathbf{B}}_{1} \mathbf{A}_{11: 2}\mathbf{B}_{1}^{T}-\mathbf{B}_{1}\mathbf{A}_{11: 2}\hat{\mathbf{B}}_{1}^T+\mathbf{B}_{1}\mathbf{A}_{11: 2}\mathbf{B}_{1}^T)/r_1$. Then define  $\mathbf H=\hat{\mathbf{B}}_{1} \mathbf{A}_{11: 2}\hat{\mathbf{B}}_{1}^{T}/r_1$, so that
$$
\mathbf H=\mathbf{H}_0+\mathbf \Delta_0,
$$
where $\mathbf \Delta_0=(\hat{\mathbf{B}}_{1} \mathbf{A}_{11: 2}\mathbf{B}_{1}^{T}+\mathbf{B}_{1}\mathbf{A}_{11: 2}\hat{\mathbf{B}}_{1}^T-\mathbf{B}_{1}\mathbf{A}_{11: 2}\mathbf{B}_{1}^T)/r_1$. Here, $\mathbf \Delta_0$  has $M_0$ dominated  eigenvalues, indicating the spikes, and thus the matrix $\mathbf H$  represents  a   perturbation of rank $M_0$ on the matrix $\mathbf H_0$. 
We demonstrate the spiked structure and  the eigenvalue  properties  of $\mathbf H$  in the supplementary material.
Combining the formula \eqref{wishart}, the matrix $\mathbf{H}\mathbf{G}^{-1}$ corresponds exactly to the spiked Fisher matrix  with $M_0$ spikes defined in \eqref{F_def}   under the null. 

To test the number of significant regression variables, equivalent to testing the number of  spikes, we use the perturbed matrix $\mathbf H$ %instead of $\mathbf H_0$ 
to design the modified LRT statistic, given by
\begin{align} \label{Lambda1}
	\Lambda_{n}^{*}=\left|\mathbf{I}+\frac{r_{1}}{n-r}\mathbf H\mathbf G^{-1} \right|^{-1}.
\end{align}

Denote the eigenvalues of $\mathbf{H}\mathbf{G}^{-1}$ in \eqref{Lambda1} in  descending order as below:
$$l_1 \ \geqslant \ l_2 \ \geqslant \ \cdots  \geqslant \ l_p, $$
where the largest $M_0$ eigenvalues are  sample spiked eigenvalues.
Then, the test statistic for the hypothesis $\eqref{H2}$ can be constructed as a   partial linear spectral statistic defined in \eqref{newtest2}, with the function $f$ using the traditional log transformation like \cite{Bai2013}, i.e.
$$-\log {\Lambda}_{n}^*-\sum_{i=1}^{M}\log(1+\frac{c_2}{c_1}l_i).$$
Its CLT is established as follows, and the computation of mean and variance can be found in \cite{Bai2013}.

\begin{theorem}\label{Th2}
	For the testing problem \eqref{H2} in the large-dimensional linear regression model in \eqref{linear},  $\Lambda_n^*$ is defined as in the formula \eqref{Lambda1}, and it is assumed that $p/r_1 \rightarrow c_{1} \in(0,\infty)$, $ p/(n-r) \rightarrow c_{2} \in(0,1)$ as $\min(p,r_1,n-r)  \rightarrow \infty$. Then, under the null, we have
	\begin{align}
		T_{l}=\nu_{l}^{-\frac{1}{2}}\left[-\log {\Lambda}_{n}^*-\sum_{i=1}^{M}\log(1+\frac{c_2}{c_1}l_i) -d_{l1} +d_{l2}-\mu_{l}\right] \Rightarrow \mathcal{N}(0, 1),
	\end{align} \label{test_linear}

	where 
	$$d_{l1}=p \int \log(1+\frac{c_2}{c_1}x) dF^{(c_{n1},c_{n2},H_n)}(x) $$ can be obtained by Algorithm 1 in the supplementary material, 
	$$d_{l2}=\sum\limits_{i=1}^{M} \log(1+\frac{c_2}{c_1}\psi_i), 
	\mu_{l}=\frac{1}{2} \log \frac{\left(c^{2}-d^{2}\right) h^{2}}{\left(c h-c_{2} d\right)^{2}},
	\text{and} \ \nu_{l}=2 \log \left(\frac{c^{2}}{c^{2}-d^{2}}\right).
	$$
	Here,
	$$
	h=\sqrt{c_1+c_2-c_1c_2}, \quad
	a, b =\frac{\left(1 \mp h\right)^{2}}{\left(1-c_2\right)^{2}},\quad a< b,
	$$
	and
	$$
	c,d=\frac{1}{2}\left(\sqrt{1+\frac{c_2}{c_1} b} \pm \sqrt{1+\frac{c_2}{c_1} a}\right), \quad c> d. 
	$$
	\end{theorem}
	
	Based on Theorem~\ref{Th2}, we can identify the number of significant variables through a sequential testing procedure. We initiate with $M_0=1$ and proceed with the sequential test until the statistic first falls within the acceptance domain. The number of significant variables is then determined by the position at which this occurs.

\section{Application to  change point detection}\label{sec4}
Change point detection can identify mutations of some statistical characteristics in sequence data. As a processing tool, it has proven useful in many applications such as  DNA segmentation, econometrics, and disease demography.  In this section,  we will employ a hypothesis testing method to detect the change points.   The test of the existence of change points will be transformed into testing whether there are spikes in the Fisher matrices constructed in the following way.

Given a $p$-dimensional sequence data $\mathbf X=\left(\mathbf{x}_{1},\mathbf{x}_{2} \cdots, \mathbf{x}_T\right)$ of length $T$ with both  change points and additive outliers, and the locations of which are unknown in practice.
Additive outliers affect only a single or few specific observations, whereas change points occur over a relatively long period.  Assuming that the change points occur at the time $t_c$, which  needs to be correctly detected in practice  while avoiding the influence of additive outliers.

As one of the fundamental tools, we use the point-by-point sliding time window  for change point detection, which can be described below.
At each step of sliding, the first sample in the window is removed, and a new sample is added at the end. Similar to  \cite{Chen2022}, we partition the samples within each time window $W_j, j=1,2,\cdots$ into two groups, with their respective population covariance matrices denoted as $\mathbf \Sigma_{j}^{(1)}$ and $\mathbf \Sigma_j^{(2)}$. 
Let  the true number of spikes in the matrix $(\boldsymbol{\Sigma}_{j}^{(1)})^{-1} \boldsymbol{\Sigma}_{j}^{(2)}$ in the  window $W_j$ be denoted by $M_j$, $j=1,2,\cdots$.  
When there are no anomalies in $W_j$, we have $(\boldsymbol{\Sigma}_{j}^{(1)})^{-1}\boldsymbol{\Sigma}_{j}^{(2)}=\mathbf{I} $, thus $M_j=0$. Conversely,  when anomalies are introduced in $W_j$,   $M_j~\neq0$.
To detect the change point in the time window $W_j$, we  consider the following hypothesis
\begin{align} \label{H5}
	\mathcal{H}_{0}: M_j=0 \quad  v.s.  \quad \mathcal{H}_{1}: M_j\neq 0, j=1,2,\cdots.
\end{align}
This is a degenerate hypothesis for the problem \eqref{H1}.

 Denote $\mathbf X_{j}^{(1)}$ and  $\mathbf X_{j}^{(2)}$ as the two distinct  group of samples in the window $W_j$, with the sample size $q_{j1}$ and $q_{j2}$, respectively. Then, the  corresponding unbiased sample covariance matrices can be denoted by $\mathbf S_{j}^{(1)}$ and  $\mathbf S_{j}^{(2)}$. Using the point-by-point sliding time window method, multiple tests are required, making the method inherently complex. To manage this complexity, a simple form of the test statistic, $\mathrm{tr}[(\mathbf{S}_{j}^{(1)})^{-1}\mathbf{S}_{j}^{(2)}]$, is employed. The key observation is that the Fisher matrix $(\mathbf{S}_{j}^{(1)})^{-1} \mathbf{S}_{j}^{(2)}$ has no sample spikes  under the null hypothesis, while  it has spikes under the alternative hypothesis.

To develop the CLT of the test statistic, we introduce some notations. 
Normalize the first group of data  $\mathbf X_{j}^{(1)}$ as $\boldsymbol \xi_{j}^{(1)}$, and  the second  group of data $\mathbf X_{j}^{(2)}$ as $ \boldsymbol \xi_{j}^{(2)}$. 
Further, define  $\phi_{j}^{(1)}=E\left|\xi_{j,11}^{(1)}\right|^{4}$, $\phi_{j}^{(2)}=E\left|\xi_{j,11}^{(2)}\right|^{4}$, where $\xi_{j,11}^{(1)}$ and $\xi_{j,11}^{(2)}$ are  the $(1,1)$-th entries  of the matrices $\boldsymbol  \xi_{j}^{(1)}$ and $\boldsymbol \xi_{j}^{(2)}$, respectively. 
Then, the asymptotic distribution of the test statistic $\mathrm{tr}[(\mathbf{S}_{j}^{(1)})^{-1}\mathbf{S}_{j}^{(2)}]$ is established as follows, which is a natural consequence of Theorem \ref{Th1}, and the calculation can refer to Example \ref{ex1}.
\begin{corollary}\label{Th3}
	For the testing problem \eqref{H5} in each  window $W_j,j=1,2,\cdots$,  assuming that $p/q_{j1} \rightarrow c_{j}^{(1)} \in(0,1)$, $ p/q_{j2} \rightarrow c_{j}^{(2)} \in (0,+\infty)$ as  $\min(p, q_{j1}, q_{j2}) \rightarrow \infty$. Then, under the null, we have
	\begin{align}\label{Tj}
		T_{j}=\nu_{x}^{-\frac{1}{2}}\left\{\mathrm{tr}[(\mathbf{S}_{j}^{(1)})^{-1} \mathbf{S}_{j}^{(2)}]-d_{x}-\mu_{x}\right\} \Rightarrow \mathcal{N}(0,1), 
	\end{align}
	where
	$$
	\begin{aligned}
		d_{x}&=\frac{p}{1-c_{j}^{(1)}}, \\
		\mu_{x}&=\frac{qc_{j}^{(1)}}{(1-c_{j}^{(1)})^2}+\frac{\beta_{j}^{(1)}c_{j}^{(1)}}{1-c_{j}^{(1)}},\\ 
		\nu_{x}&=\frac{(q+1)h_j^2}{(1-c_{j}^{(1)})^4}+\frac{\beta_{j}^{(1)}c_{j}^{(1)}+\beta_{j}^{(2)}c_{j}^{(2)}}{(1-c_{j}^{(1)})^2}.
	\end{aligned}
	$$
	Here,  $h_j^2=(c_j^{(1)})^2+(c_j^{(2)})^2-c_j^{(1)}c_j^{(2)}$,  $\beta_{j}^{(1)}={\phi_{j}^{(1)}-q-2}$ and $\beta_{j}^{(2)}={\phi_{j}^{(2)}-q-2}$ with $q=1$ for real case and $q=0$ for complex. 	
\end{corollary}

Based on Corollary~\ref{Th3}, we design  an algorithm to detect the location  $t_c$ of the change point, as elaborated in Algorithm \ref{alg2}. Note that the criterion for determining whether a change point is detected is that abnormal points are found in $s$ consecutive windows.

\begin{algorithm} 	
	\caption{Change Point Detection Based on Hypothesis Testing.}
	\renewcommand{\algorithmicrequire}{\textbf{Input:}}
	\renewcommand{\algorithmicensure}{\textbf{Output:}}
	\begin{algorithmic}[1]
		\REQUIRE $\mathbf     X=\left(\mathbf{x}_{1}, \mathbf{x}_2,\cdots, \mathbf{x}_T\right)$; $q_{11}, q_{12}$; $s$; an  empty set $\mathcal{T}$. 
		\ENSURE The location  $t_c$ of change point.
		\STATE $j=0$; 
		\WHILE {There are no $s$ consecutive numbers in the set $\mathcal{T}$}
		\STATE  $j=j+1$;
		\STATE  Partition the samples in the  window $W_j$ into two groups, and compute the respective sample covariance matrices as $\mathbf{S}_{j}^{(1)}$ and  $\mathbf{S}_{j}^{(2)}$;
		\STATE  Calculate $T_j$ in the window $W_j$ according to \eqref{Tj};
		\IF{$T_j$ falls outside the rejection region} 
		\STATE   Slide the Window to $ W_{j+1}$;
		\ELSE
		\STATE    Delete the sample $x_{t_0}$ at the end of $W_j$, where $t_0$  is the location of this sample in the  data $\mathbf X$ ;
		\STATE    Add $t_0$ to the set $\mathcal{T}$;
		\STATE    Slide the Window to $ W_{j+1}$;
		\ENDIF
		\ENDWHILE
		\STATE  Obtain the first value of the $s$ consecutive numbers in $\mathcal{T}$ as the location $t_c$ where the change point occurs.
	\end{algorithmic}
	\label{alg2}
\end{algorithm}

\begin{remark}
 	A critical step in Algorithm \ref{alg2} is the removal of abnormal samples in step 9, which is essential for avoiding the influence of additive outliers on change point detection.	
 	Without this removal, the newly introduced additive outliers will exist in multiple consecutive time windows until they slip out of the window,  which can also produce consecutive $s$ anomalous results  and thus lead to  false detection of change points.
 	Furthermore, by the removal of abnormal points, for the lengths of two groups of data in all  windows,  we have $q_{11}=q_{21}=\cdots=q_{j1}=\cdots$ and $q_{12} \geq q_{22}\geq \cdots \geq q_{j2} \geq\cdots$.
 \end{remark}

\section{Numerical and empirical studies}\label{sec5}
%In this section,  numerical and empirical studies are conducted to evaluate the performance of our proposed methods.
\subsection{Evaluation for the CLT proposed in Theorem 2.1}\label{sec5.1}
To verify the effectiveness of our established CLT in Theorem \ref{Th1},  the following two models are considered.\\
\textbf{Model 1:} Assuming that the matrix $ \mathbf{R}_{p}^{*} \mathbf{R}_{p}$ is  non-diagonal, with $\mathbf{\Sigma_{2}}=\mathbf I_p$ and $\mathbf{\Sigma_{1}}=\mathbf U \text{diag}(10,8,8,6,1,\cdots,1)\mathbf U^*$, where $\mathbf U$ is an orthogonal matrix. \\
\textbf{Model 2:} Assuming that the matrix $ \mathbf{R}_{p}^{*} \mathbf{R}_{p}$ is  diagonal, with $\mathbf{\Sigma_{2}}=\mathbf I_p$ and  $\mathbf \Sigma_1=\text{diag}(36,25,25,16,2,\cdots,2,1,\cdots,1)$. Here, ``2'' has multiplicity $p/2-4$, and ``1'' has multiplicity $p/2$. 

In both models, the true number of spikes  is $M_0=4$. Moreover, in Model 1, the LSD of $\mathbf{\Sigma_{1}}\mathbf{\Sigma_{2}}^{-1}$  is $H(t)=\delta_{\{1\}}$, while in Model 2, the LSD is $H(t)=0.5\delta_{\{1\}}+0.5\delta_{\{2\}}$. In particular,  the mean term $\mu_{f,H}$ and variance term $\nu_{f,H}$ in Model 2 can be obtained numerically by the procedure proposed in \cite{zheng2017}.
 As a robustness check of our method, the samples $\{x_{ij}\}$ and $\{y_{il}\}$ in each model are generated from the following two populations.\\
\textbf{Gaussian Assumption:} Both $\{x_{ij}\}$ and $\{y_{il}\}$ are  i.i.d. samples from $\mathcal N(0, 1)$.\\
\textbf{Gamma Assumption:} Both $\{x_{ij}\}$ and $\left\{y_{il}\right\}$ are  i.i.d.  samples from \\$ \left\{ Gamma(2,1)-2\right\} / \sqrt2$.\\

In such settings, four scenarios are considered.  We select $f(x)= x$ and $f(x)=\log x$ as  examples to conduct simulations.   For each scenario, the value of dimensionality  is set to $p=100, 200, 400$. The ratio pairs $(c_1, c_2)$ are set to $(0.8, 0.2), (2, 0.2), (0.2, 0.5) $ for $f(x)= x$, and $(0.5, 0.2), (0.4, 0.5), (0.2, 0.5) $ for $f(x)=\log x$, respectively. 

Tables \ref{M1_log} and \ref{M2_log} report the empirical sizes and powers  of rejecting the null hypothesis \eqref{H1}  over 1,000 replications in Models 1 and 2 at a significance level $\alpha = 0.05$ when $f(x)=\log(x)$. 
For space considerations, similar results for the function $f(x)=x$ are presented in the supplementary material. Under different models and  population assumptions, the empirical sizes of our proposed test statistic are approximately equal to the significance level when the null hypothesis is true. 
The further the alternative hypothesis is from the null hypothesis, the higher the power is.  From these results, it can be inferred that the true value $M_0$ aligns with the location of the first local minimum of  empirical sizes.
%which illustrates that our test has a good performance. 
Furthermore, Figures \ref{fig:M1_log} and \ref{fig:M2_log} show the empirical density histograms of our proposed test statistic under the null for  cases where $p=200$. These figures suggest that the proposed statistic is asymptotically normal when the null hypothesis is true, further validating the effectiveness of our method.

\begin{table}[htbp]
	\centering
	\caption{Empirical size of rejecting the null hypothesis \eqref{H1} in Model 1 when $f(x)=\log x$.}
	\small
	\renewcommand\arraystretch{1.2}
	\setlength{\tabcolsep}{1.3mm}
	\begin{tabular}{lllllll}
		\hline
		Values of $M_0$& \multicolumn{1}{l}{$M_0=1$} & \multicolumn{1}{l}{$M_0=2$} & \multicolumn{1}{l}{$M_0=3$} & \multicolumn{1}{l}{$M_0=4$} & \multicolumn{1}{l}{$M_0=5$} & \multicolumn{1}{l}{$M_0=6$}  \\
		\hline
		\multicolumn{7}{c}{ Gaussian population } \\
		\hline
		$p=100,c_1=0.5,c_2=0.2$ &0.993  & 0.785  & 0.251  & \textbf{0.050 }  & 0.083  & 0.154  \\
		$p=200,c_1=0.5,c_2=0.2$ & 0.988  & 0.790  & 0.242  & \textbf{0.051}  & 0.090  &0.183  \\
		$p=400,c_1=0.5,c_2=0.2$ &0.992  & 0.803  & 0.265  & \textbf{0.054} & 0.079  &0.188 \\
		\hline
		$p=100,c_1=0.4,c_2=0.5$ & 0.964  & 0.682  & 0.209  & \textbf{0.043}  & 0.074  & 0.152 \\
		$p=200,c_1=0.4,c_2=0.5$ & 0.97  & 0.658 & 0.198  & \textbf{0.047}  &0.090  &0.253   \\
		$p=400,c_1=0.4,c_2=0.5$ & 0.964  & 0.694  & 0.211  & \textbf{0.050}  & 0.122  &0.272  \\
		\hline
		$p=100,c_1=0.2,c_2=0.5$ & 0.990  & 0.774  & 0.263  & \textbf{0.053}  & 0.106  &0.204  \\
		$p=200,c_1=0.2,c_2=0.5$ & 0.990  & 0.769  &0.246 & \textbf{0.051 } & 0.113  &0.328 \\
		$p=400,c_1=0.2,c_2=0.5$ & 0.989  & 0.800  & 0.260  & \textbf{0.052}  & 0.134 & 0.352 \\
		\hline
		\multicolumn{7}{c}{Gamma population } \\
		\hline
		$p=100,c_1=0.5,c_2=0.2$ & 0.863  &0.482 &0.132 &  \textbf{0.045}&0.060& 0.084 \\
		$p=200,c_1=0.5,c_2=0.2$ & 0.826  &0.480 &0.140   &\textbf{0.044} & 0.069 &0.118 \\
		$p=400,c_1=0.5,c_2=0.2$ & 0.845  &0.471 &0.139 &\textbf{0.045}& 0.070& 0.133 \\
		\hline
		$p=100,c_1=0.4,c_2=0.5$& 0.720  &0.363 &0.116&\textbf{0.043} &0.064& 0.122 \\
		$p=200,c_1=0.4,c_2=0.5$ & 0.755 &0.397&0.137 &\textbf{0.051} &0.075 & 0.149 \\
		$p=400,c_1=0.4,c_2=0.5$ & 0.763  &0.404& 0.117 &\textbf{0.052} &0.090  & 0.159\\
		\hline
		$p=100,c_1=0.2,c_2=0.5$& 0.836 &0.456 &0.146 &\textbf{0.042} &0.068 & 0.121 \\
		$p=200,c_1=0.2,c_2=0.5$ & 0.834  &0.434 & 0.128  &\textbf{0.047}& 0.084 &0.187 \\
		$p=400,c_1=0.2,c_2=0.5$ & 0.854 &0.464 &0.159   &\textbf{0.052} &0.096 &0.190 \\
		\hline	
		
	\end{tabular}%
	\label{M1_log}%
\end{table}%

\begin{figure}[htbp]
	\centering
	\subfigure{\includegraphics[width=1.55in]{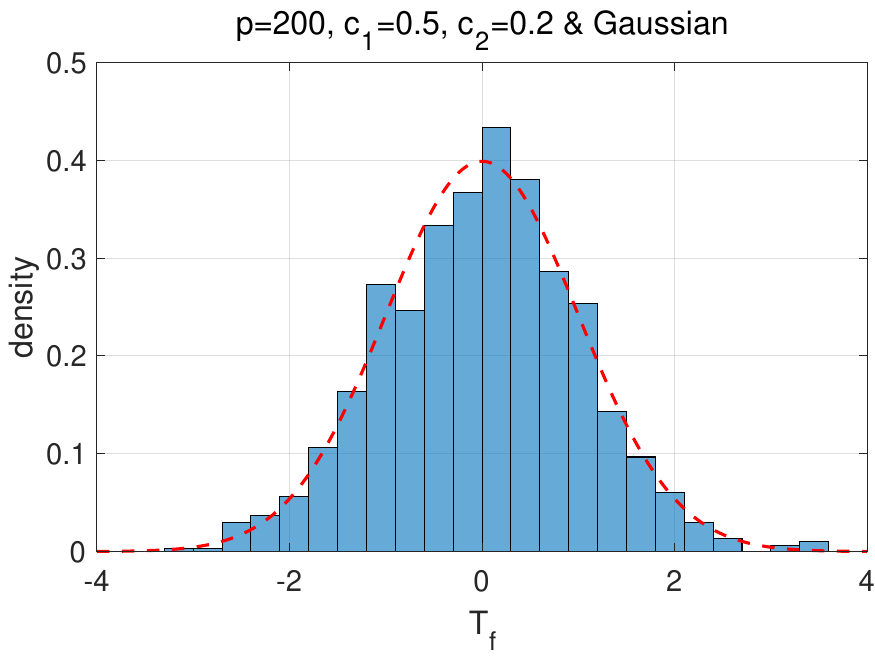}}
	\subfigure {\includegraphics[width=1.55in]{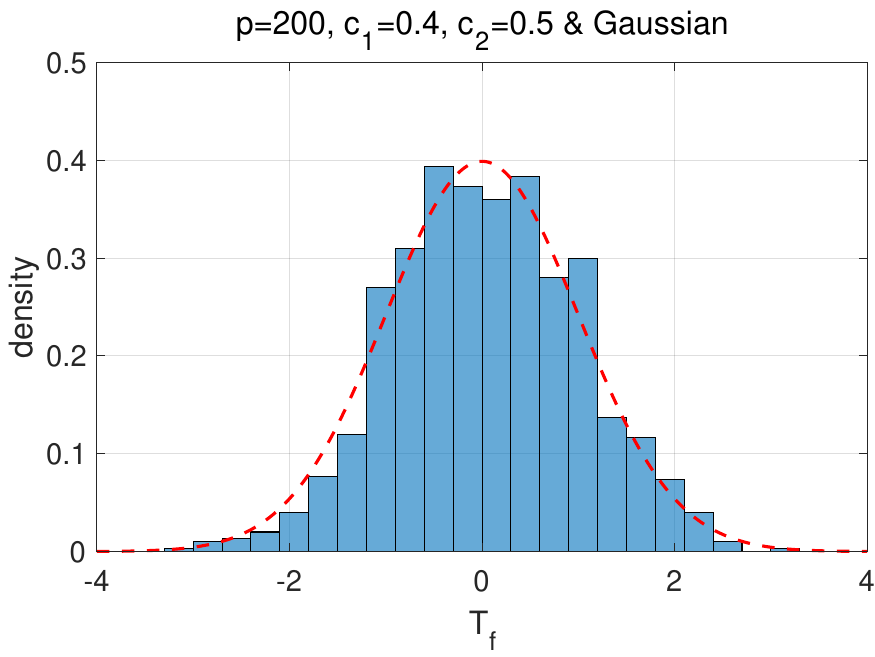}}
	\subfigure {\includegraphics[width=1.55in]{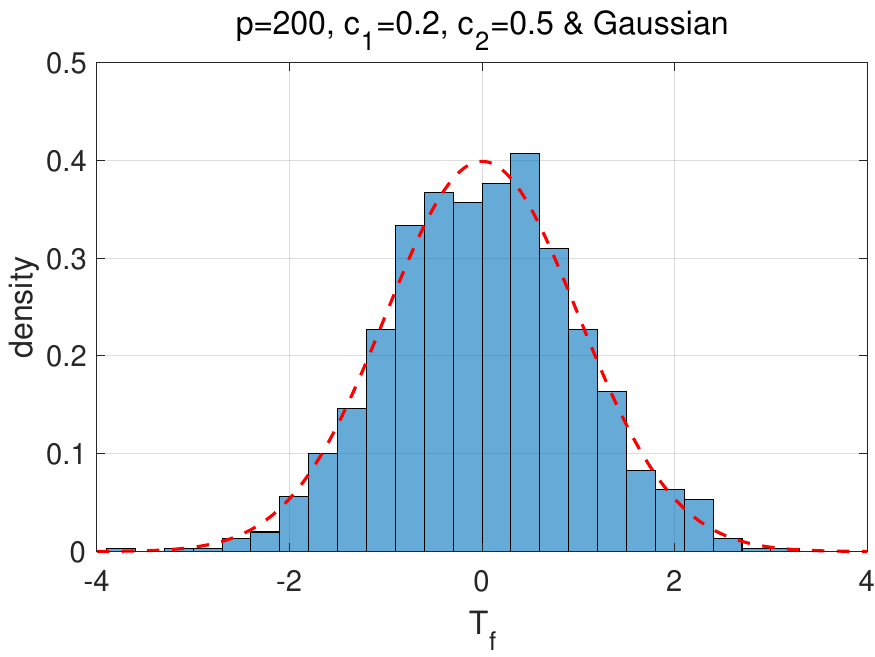}}
	
	\subfigure {\includegraphics[width=1.55in]{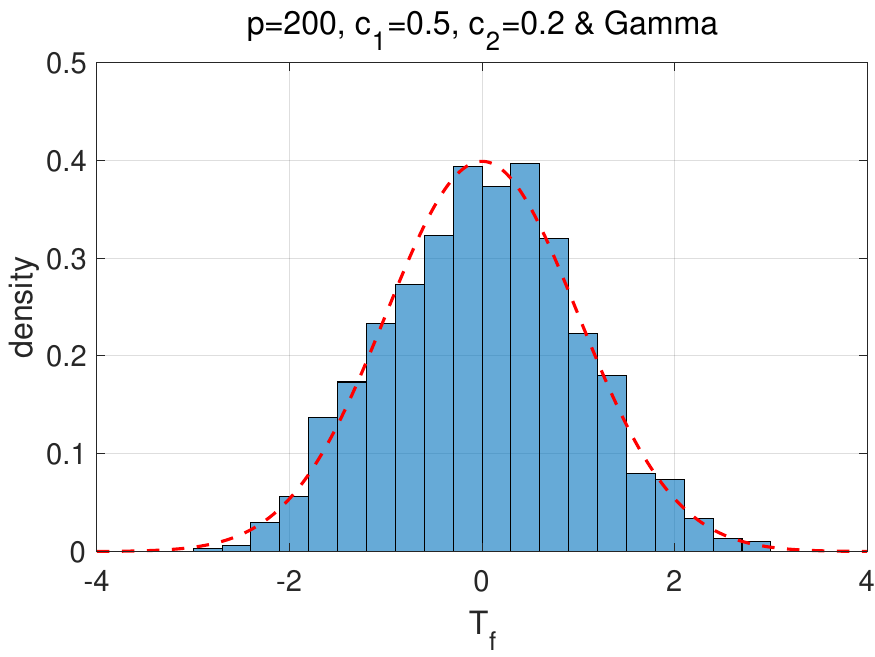}}
	\subfigure {\includegraphics[width=1.55in]{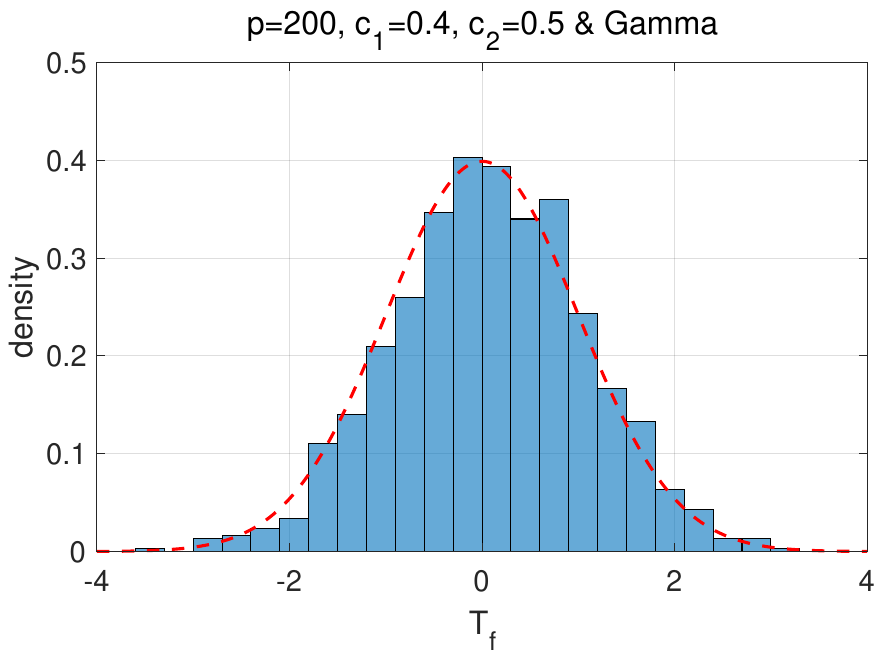}}
	\subfigure {\includegraphics[width=1.55in]{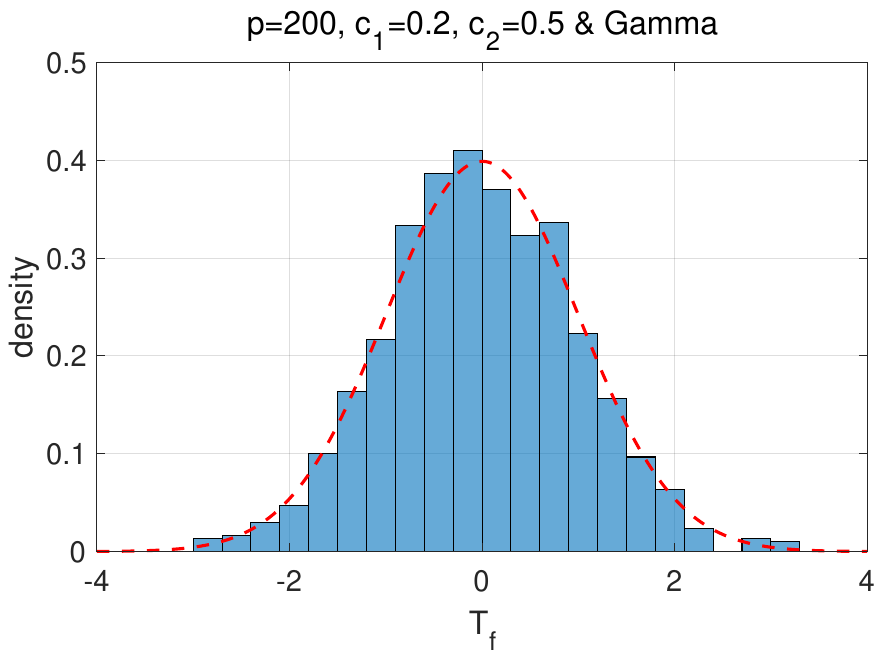}}
	\vspace{-1em}
	\caption{Empirical distribution of $T_{f}$ under the null in Model 1 when  $f(x)=\log x$. }
	\label{fig:M1_log}
\end{figure}

\begin{table}[htbp]
	\centering
	\caption{Empirical size of rejecting the null hypothesis \eqref{H1} in Model 2 when $f(x)=\log x$.}
	\small
	\renewcommand\arraystretch{1.2}
	\setlength{\tabcolsep}{1.2mm}
	\begin{tabular}{lllllll}
		\hline
		Values of $M_0$& \multicolumn{1}{l}{$M_0=1$} & \multicolumn{1}{l}{$M_0=2$} & \multicolumn{1}{l}{$M_0=3$} & \multicolumn{1}{l}{$M_0=4$} & \multicolumn{1}{l}{$M_0=5$} & \multicolumn{1}{l}{$M_0=6$}  \\
		\hline
		\multicolumn{7}{c}{Gaussian population} \\
		\hline
		$p=100,c_1=0.5,c_2=0.2$ & 0.999  & 0.900  & 0.290  & \textbf{0.045 }  & 0.374  & 0.915  \\
		$p=200,c_1=0.5,c_2=0.2$ & 1  &0.913  & 0.311  & \textbf{0.044}  & 0.328  &0.870  \\
		$p=400,c_1=0.5,c_2=0.2$ &1  & 0.938  & 0.359 & \textbf{0.045} & 0.282  &0.872 \\
		\hline
		$p=100,c_1=0.4,c_2=0.5$ & 0.995 & 0.808 & 0.250 & \textbf{0.041}  & 0.690  & 0.996\\
		$p=200,c_1=0.4,c_2=0.5$ & 0.999  &0.858 & 0.308  & \textbf{0.051}  &0.617  &0.999   \\
		$p=400,c_1=0.4,c_2=0.5$ & 0.997  & 0.856  & 0.289  & \textbf{0.050}  & 0.678  &0.995  \\
		\hline
		$p=100,c_1=0.2,c_2=0.5$ &0.999  & 0.888  & 0.294  & \textbf{0.052}  & 0.693  &0.997  \\
		$p=200,c_1=0.2,c_2=0.5$ &0.999  &0.918  & 0.334& \textbf{0.052 } & 0.627  &0.996 \\
		$p=400,c_1=0.2,c_2=0.5$ & 1  & 0.917  & 0.324  & \textbf{0.056}  & 0.825 & 1\\
		\hline
		\multicolumn{7}{c}{Gamma population } \\
		\hline
		$p=100,c_1=0.5,c_2=0.2$ & 0.929  &0.577 &0.157 &  \textbf{0.040} &0.213 & 0.616 \\
		$p=200,c_1=0.5,c_2=0.2$ & 0.955  &0.649 &0.200   &\textbf{0.048} & 0.223 &0.654 \\
		$p=400,c_1=0.5,c_2=0.2$ & 0.949  &0.629 &0.191 &\textbf{0.053}& 0.220& 0.661 \\
		\hline
		$p=100,c_1=0.4,c_2=0.5$& 0.892  &0.507 &0.147&\textbf{0.045} &0.352& 0.897 \\
		$p=200,c_1=0.4,c_2=0.5$ & 0.880  &0.516 &0.153 &\textbf{0.047} &0.257 & 0.810\\
		$p=400,c_1=0.4,c_2=0.5$ & 0.878  &0.491 &0.146  &\textbf{0.054} &0.244  & 0.824\\
		\hline
		$p=100,c_1=0.2,c_2=0.5$ & 0.938  &0.615  &0.182   &\textbf{0.043}  &0.466 & 0.954 \\
		$p=200,c_1=0.2,c_2=0.5$ & 0.958  &0.606 & 0.183  &\textbf{0.046}& 0.492 &0.955 \\
		$p=400,c_1=0.2,c_2=0.5$ & 0.961  &0.662 &0.213 &\textbf{0.049} & 0.376 &0.944 \\
		\hline	
		
	\end{tabular}%
	\label{M2_log}%
\end{table}%

\begin{figure}[htbp]
	\centering
	\subfigure{\includegraphics[width=1.55in]{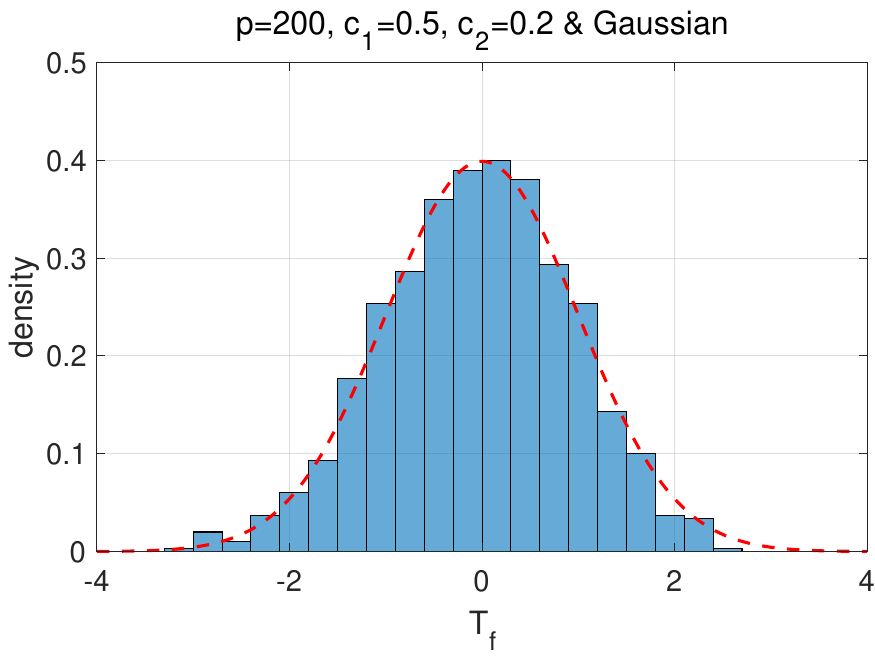}}
	\subfigure {\includegraphics[width=1.55in]{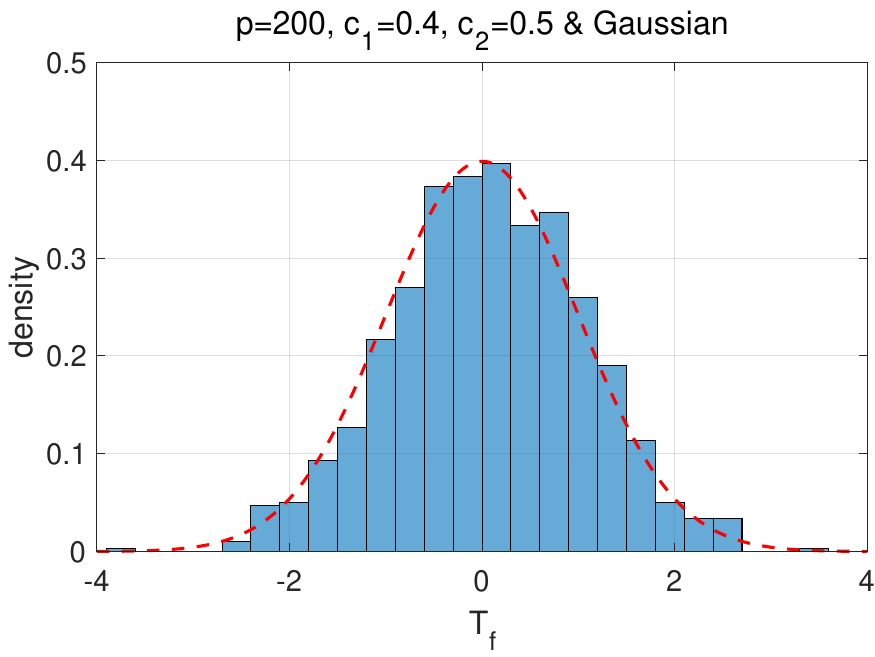}}
	\subfigure {\includegraphics[width=1.55in]{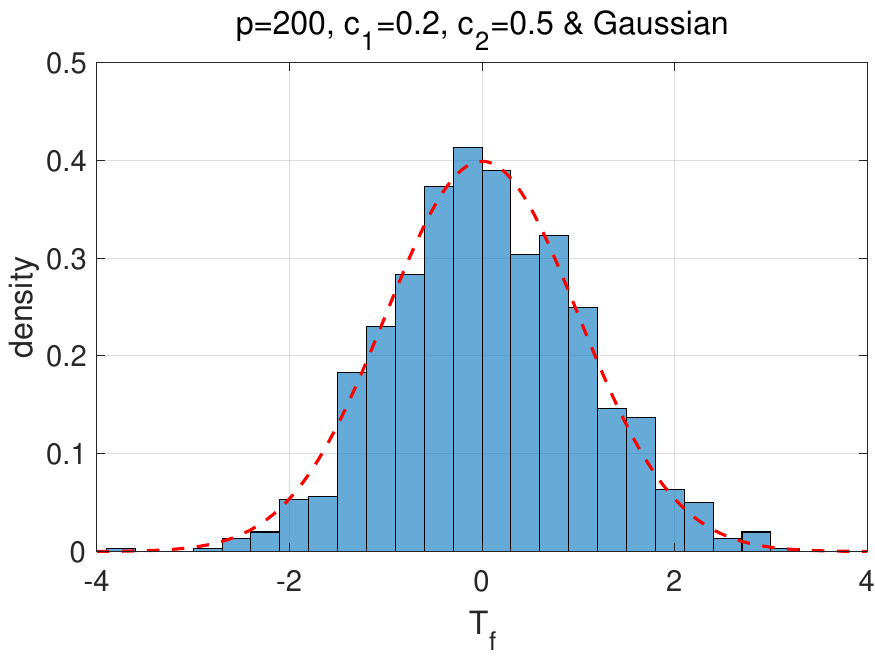}}
	
	\subfigure {\includegraphics[width=1.55in]{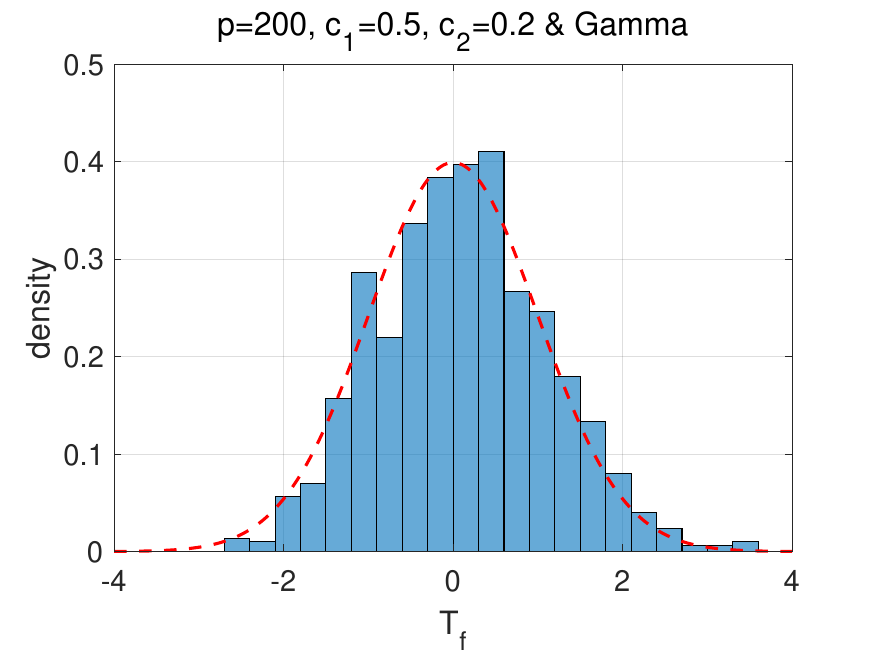}}
	\subfigure {\includegraphics[width=1.55in]{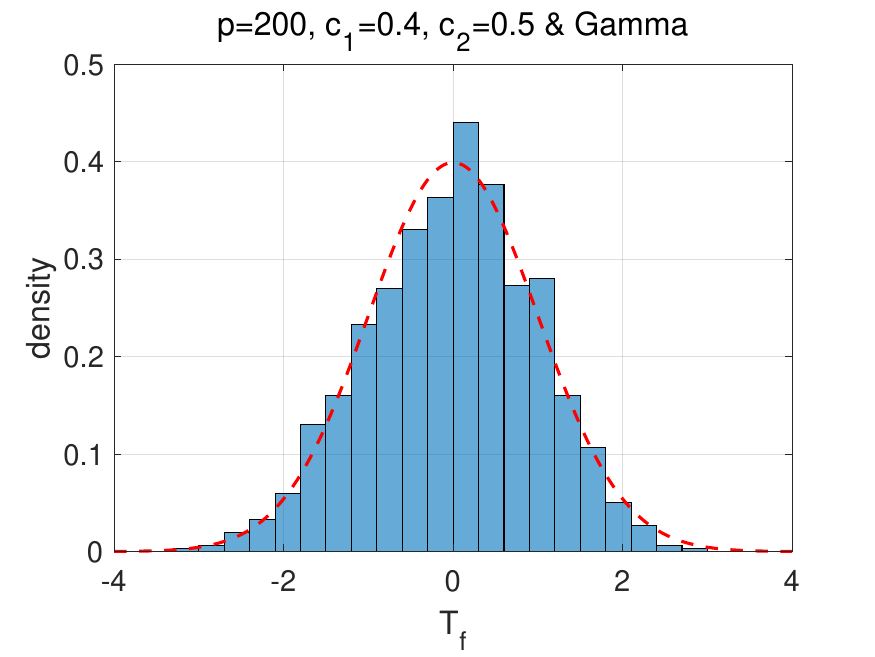}}
	\subfigure {\includegraphics[width=1.55in]{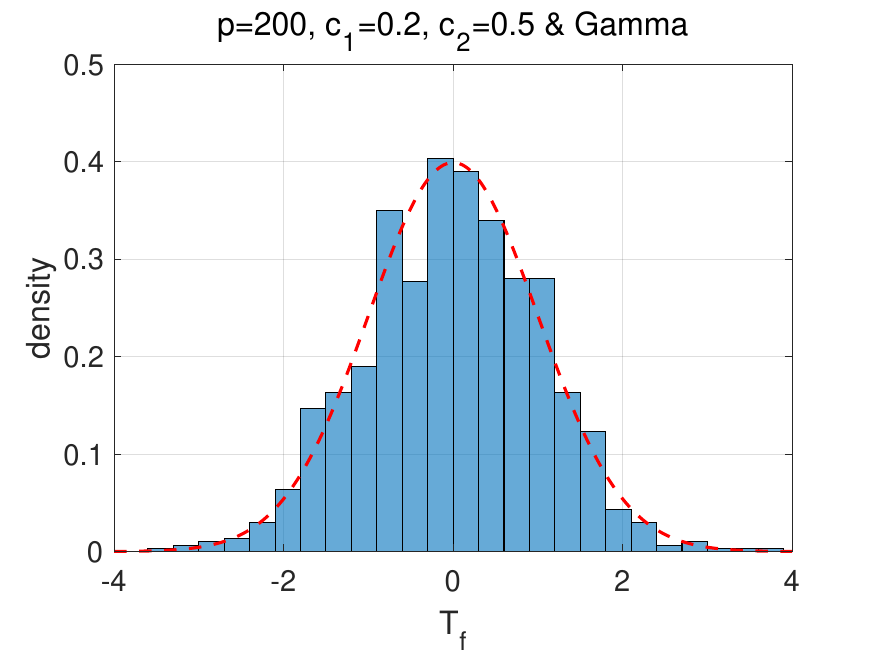}}
	\vspace{-1em}
	\caption{Empirical distribution of $T_{f}$ under the null in Model 2 when  $f(x)=\log x$. }
	\label{fig:M2_log}
\end{figure}

\subsection{Numerical studies for large-dimensional linear regression}

For the hypothesis \eqref{H2} in the large-dimensional linear regression model, we  generate data from the following form:
\begin{align*}
	\mathbf{z}_{i}=\mathbf{B} \mathbf{w}_{i}+\boldsymbol\epsilon_{i}, \quad i=1,2, \ldots, n,
\end{align*}
where the entries of regression variables  $\mathbf{w}_{i} $ are i.i.d. drawn from $\mathcal{N}(0, 1) $, the regression coefficient matrix  $\mathbf{B}$ is decomposed into $(\mathbf B_1, \mathbf B_2)$,  $\mathbf B_1=(\mathbf b_1, \mathbf b_2, \mathbf b_3, \mathbf b_4,\\ \mathbf b_5, \mathbf 0, \mathbf0,\cdots, \mathbf0)$ with the entries of $\mathbf b_i  (i=1,2,\cdots, 5)$ being i.i.d. from $\mathcal{N}(0, 1)$, and $\mathbf B_2=\mathbf 0$. Thus, $M_0=5$.
The error  $\boldsymbol\epsilon_i$ is a $p \times 1$ random vector drawn from  $\mathcal{N}_{p}(\mathbf 0,\mathbf V)$, where the population covariance matrix $\mathbf V$  is designed as the following two models.\\
\textbf{Model 3:} Assuming that the matrix $\mathbf V$ is an identity matrix.\\
\textbf{Model 4:} Assuming that the matrix $\mathbf V$ is in the form of a Toeplitz matrix defined as
$$
\mathbf V=\left(\begin{array}{ccccc}
	1 & \rho & \rho^{2} & \cdots & \rho^{p-1} \\
	\rho & 1 & \rho & \cdots & \rho^{p-2} \\
	\cdots & \cdots &\cdots &\cdots &\cdots \\
	\rho^{p-1} & \rho^{p-2} & \cdots & \rho & 1
\end{array}\right),
$$
where $\rho=0.9$.

In each scenario,  the value of $c_1$  is set to $ 0.5, 1.5$, and  $c_2$ is set to $0.2, 0.5, 0.8$. 
Reported in Tables \ref{linear1} and \ref{linear2} are  the empirical sizes and powers of our proposed test statistic $T_{l}$ in \eqref{test_linear} rejecting the null with 2,000 replications at a significance level $\alpha = 0.05$, for Models 3 and 4, respectively.

\begin{table}[htbp]
	\centering
	\caption{Empirical size of rejecting the null in Model 3.}
	\vskip-.1cm
	%\rule[0pt]{5.1cm}{0.4pt}
	\smallskip
	\small
	\renewcommand\arraystretch{1.2}
	\setlength{\tabcolsep}{0.5mm}
	\begin{tabular}{lllllll}
		\hline
		& \multicolumn{1}{l}{$M_0=1$} & \multicolumn{1}{l}{$M_0=2$} & \multicolumn{1}{l}{$M_0=3$} & \multicolumn{1}{l}{$M_0=4$} & \multicolumn{1}{l}{$M_0=5$} & \multicolumn{1}{l}{$M_0=6$}  \\
		\hline
		$p=40, n=300, r=100, r_1=80$ & 1  &1 &1 &1& \textbf{0.0365}&0.7450  \\
		$p=80, n=600, r=200, r_1=160$ & 1  &1 &1  &1  & \textbf{0.0395}& 0.7895\\
		$p=160, n=1200, r=400, r_1=320$ & 1  &1 &1  &1  & \textbf{0.0515}&0.8125\\
		\hline
		$p=50, n=300, r=200, r_1=100$ & 1  &1  &1 &0.9995 &\textbf{0.0360}& 0.7710\\
		$p=100, n=600, r=400, r_1=200$ & 1  &1 &1&1 &\textbf{0.0430}& 0.8055\\
		$p=200, n=1200, r=800, r_1=400$ & 1  &1 &1 &1&\textbf{0.0515}& 0.8305\\
		\hline
		$p=80, n=300, r=200, r_1=160$ & 1  &1   & 1&0.9790 &\textbf{0.0365}& 0.8245 \\
		$p=160, n=600, r=400, r_1=320$ & 1  &1  &1  &0.9955& \textbf{0.0405}&0.8580\\
		$p=320, n=1200, r=800, r_1=640$ &1 &1 &1 &0.9995 &\textbf{0.0480} & 0.8360\\
		\hline
		$p=60, n=350, r=50, r_1=40$ & 1  &1   & 1&1 & \textbf{0.0390}&0.6830\\
		$p=120, n=700, r=100, r_1=80$ & 1  &1  &1  &1& \textbf{0.0395}& 0.7700 \\
		$p=240, n=1400, r=200, r_1=160$ & 1  &1   &1 &1& \textbf{0.0460}&0.7650 \\
		\hline
		$p=90, n=350, r=170, r_1=60$ & 1  &1   & 1&1 & \textbf{0.0370}& 0.7555\\
		$p=180, n=700, r=340, r_1=120$ & 1  &1  &1  &1& \textbf{0.0420}& 0.7810 \\
		$p=360, n=1400, r=680, r_1=240$ & 1  &1   &1 &1& \textbf{0.0480}&0.7850\\
		\hline
		$p=120, n=350, r=200, r_1=80$ & 1  &1   & 1&1 & \textbf{0.0375}& 0.7965\\
		$p=240, n=700, r=400, r_1=160$ & 1  &1  &1  &1& \textbf{0.0420}& 0.8250\\
		$p=480, n=1400, r=800, r_1=320$ & 1  &1   &1 &1& \textbf{0.0515}&0.8300\\
		\hline
	\end{tabular}%
	\label{linear1}%
\end{table}%

\begin{figure}[htbp]
	\centering
	\subfigure{\includegraphics[width=1.55in]{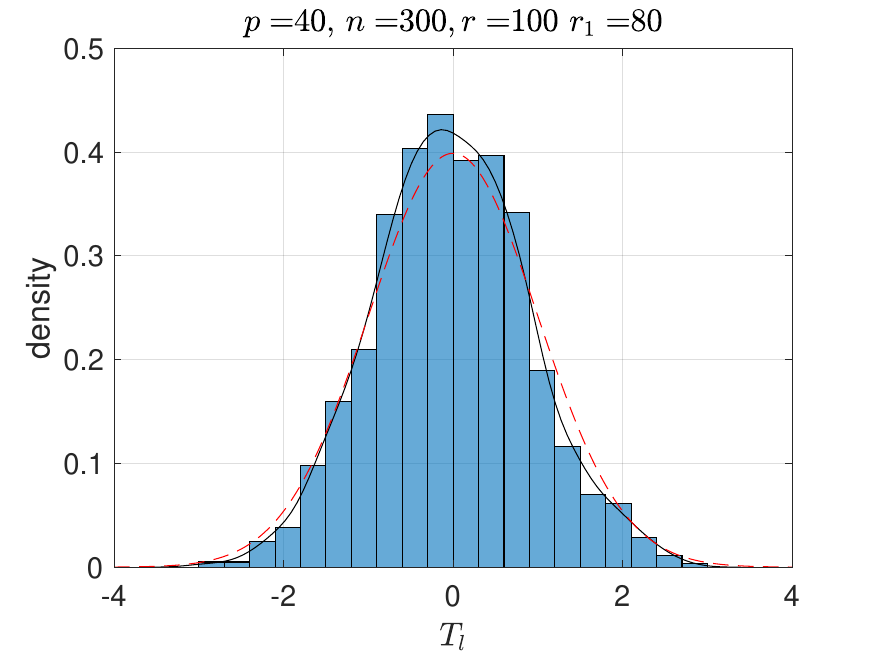}}
	\subfigure {\includegraphics[width=1.55in]{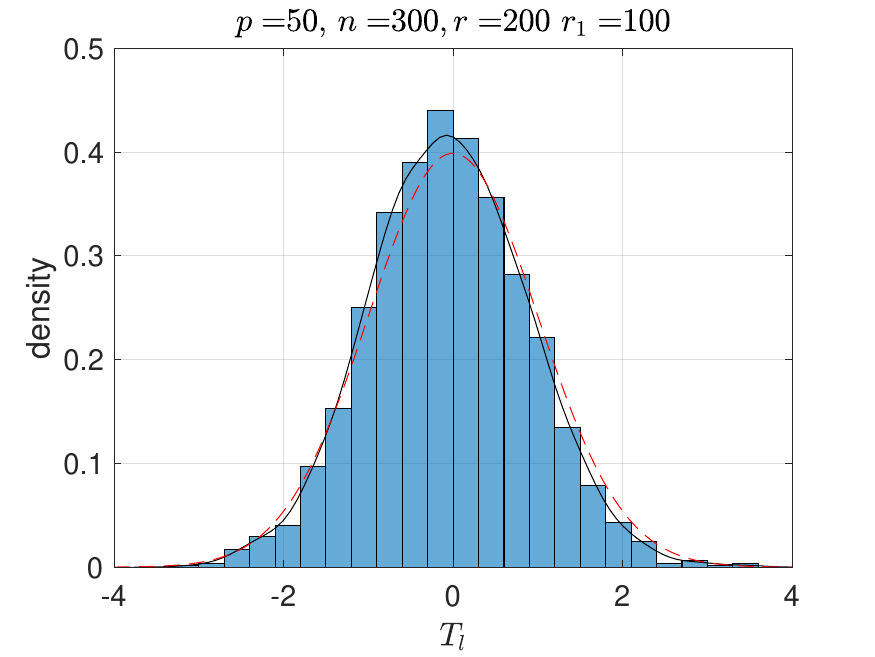}}
	\subfigure {\includegraphics[width=1.55in]{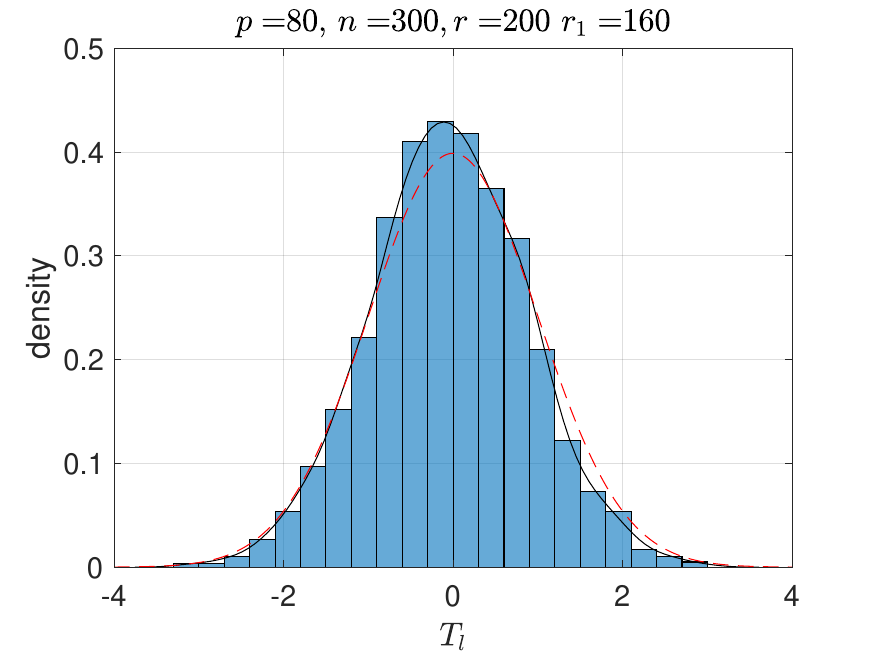}}
	
	\subfigure {\includegraphics[width=1.55in]{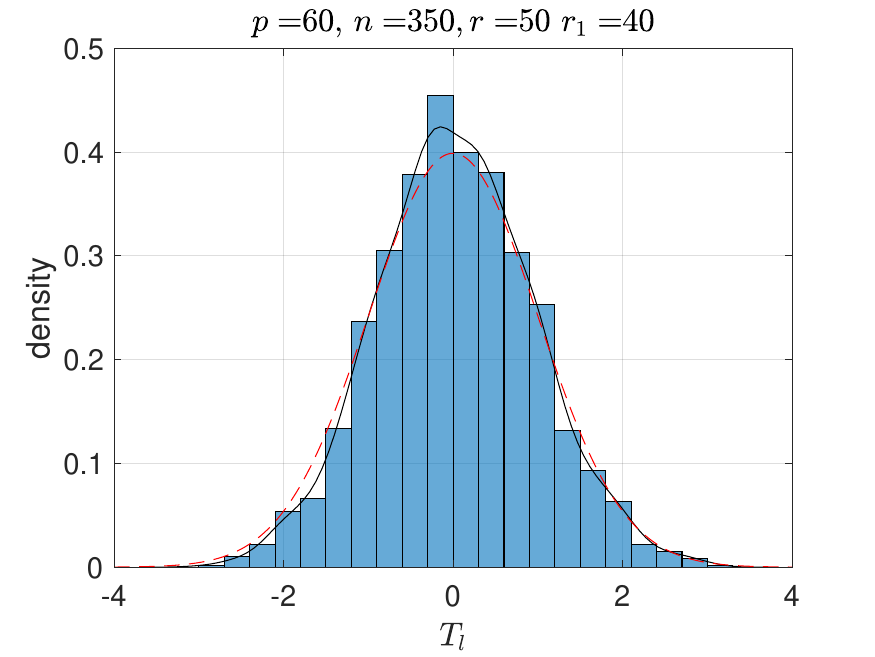}}
	\subfigure {\includegraphics[width=1.55in]{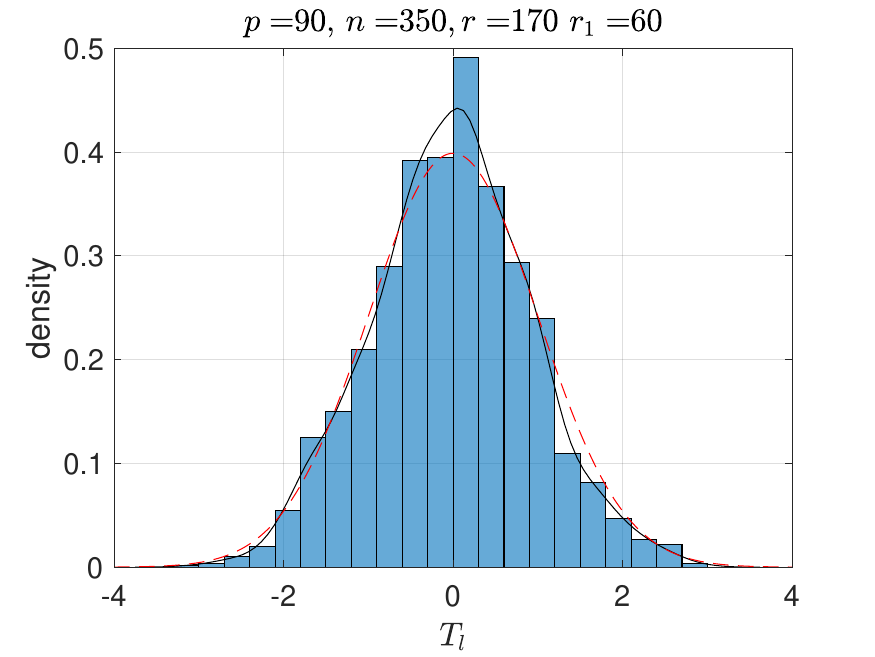}}
	\subfigure {\includegraphics[width=1.55in]{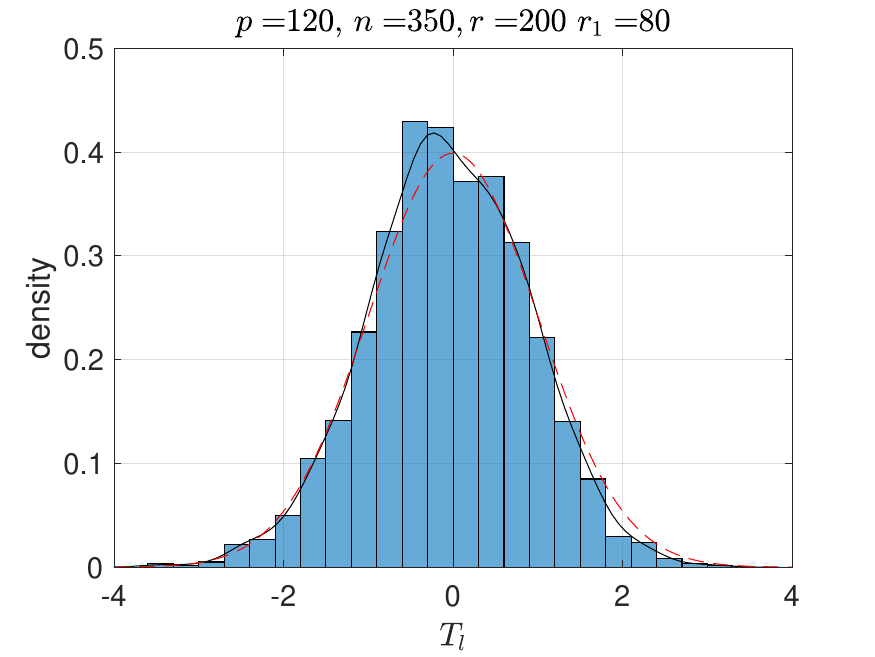}}
	\caption{Empirical distribution of $T_{l}$ under the null in Model 3. }
	\label{fig:M3}
\end{figure}

\begin{table}[htbp]
	\centering
	\vspace{-0.5em}
	\caption{Empirical size of rejecting the null  in Model 4.}
	\vskip-.1cm
	%\rule[0pt]{5.1cm}{0.4pt}
	\smallskip
	\small
	\renewcommand\arraystretch{1.2}
	\setlength{\tabcolsep}{0.5mm}
	\begin{tabular}{lllllll}
		\hline
		Values of $M_0$& \multicolumn{1}{l}{$M_0=1$} & \multicolumn{1}{l}{$M_0=2$} & \multicolumn{1}{l}{$M_0=3$} & \multicolumn{1}{l}{$M_0=4$} & \multicolumn{1}{l}{$M_0=5$} & \multicolumn{1}{l}{$M_0=6$}  \\
		\hline
		$p=40, n=300, r=100, r_1=80$ & 1  &1 &1 &1& \textbf{0.0340}&0.7410  \\
		$p=80, n=600, r=200, r_1=160$ & 1  &1 &1  &1  & \textbf{0.0440}& 0.7890\\
		$p=160, n=1200, r=400, r_1=320$ & 1  &1 &1  &1  & \textbf{0.0465}&0.7925\\
		\hline
		$p=50, n=300, r=200, r_1=100$ & 1  &1  &1 &1 &\textbf{0.0380}& 0.7890\\
		$p=100, n=600, r=400, r_1=200$ & 1  &1 &1&1 &\textbf{0.0425}& 0.8040\\
		$p=200, n=1200, r=800, r_1=400$ & 1  &1 &1 &1&\textbf{0.0550}& 0.814\\
		\hline
		$p=80, n=300, r=200, r_1=160$ & 1  &1   & 1&0.9995 &\textbf{0.0360}& 0.8315 \\
		$p=160, n=600, r=400, r_1=320$ & 1  &1  &1  &1& \textbf{0.0415}&0.8515\\
		$p=320, n=1200, r=800, r_1=640$ &1 &1 &1 &1 &\textbf{0.0520} & 0.8500\\
		\hline
		$p=60, n=350, r=50, r_1=40$ & 1  &1   & 1&1 & \textbf{0.0350}&0.7040\\
		$p=120, n=700, r=100, r_1=80$ & 1  &1  &1  &1& \textbf{0.0440}& 0.7525 \\
		$p=240, n=1400, r=200, r_1=160$ & 1  &1   &1 &1& \textbf{0.0450}&0.7825\\
		\hline
		$p=90, n=350, r=170, r_1=60$ & 1  &1   & 1&1 & \textbf{0.0345}& 0.7525\\
		$p=180, n=700, r=340, r_1=120$ & 1  &1  &1  &1& \textbf{0.0345}& 0.7770 \\
		$p=360, n=1400, r=680, r_1=240$ & 1  &1   &1 &1& \textbf{0.0475}&0.7920\\
		\hline
		$p=120, n=350, r=200, r_1=80$ & 1  &1   & 1&1 & \textbf{0.0340}& 0.7980\\
		$p=240, n=700, r=400, r_1=160$ & 1  &1  &1  &1& \textbf{0.0410}& 0.8130\\
		$p=480, n=1400, r=800, r_1=320$ & 1  &1   &1 &1& \textbf{0.0445}&0.8355\\
		\hline
	\end{tabular}%
	\label{linear2}%
\end{table}%
\begin{figure}[htbp]
	\vspace{-1em}
	\centering
	\subfigure{\includegraphics[width=1.55in]{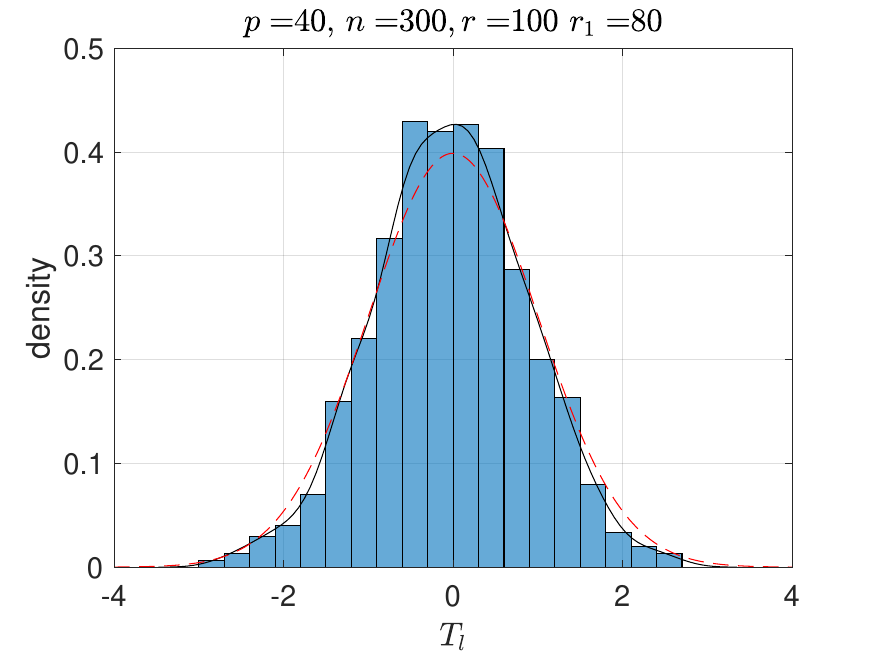}}
	\subfigure {\includegraphics[width=1.55in]{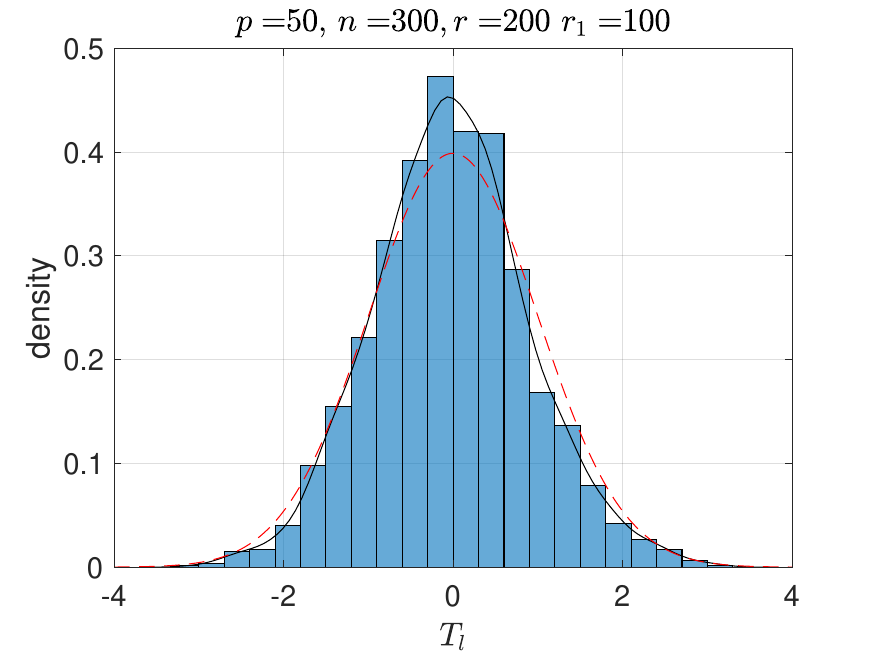}}
	\subfigure {\includegraphics[width=1.55in]{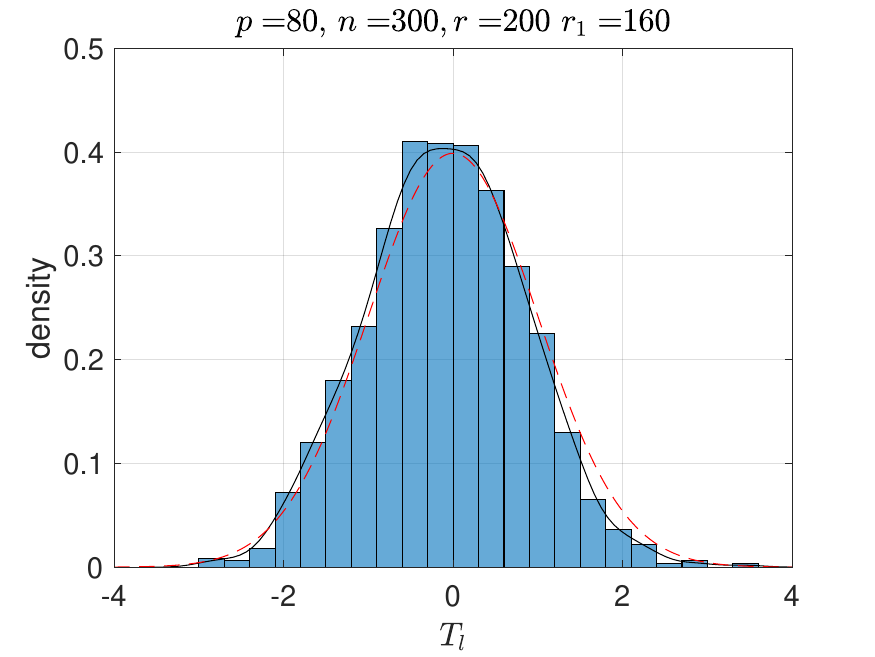}}
	
	\subfigure {\includegraphics[width=1.55in]{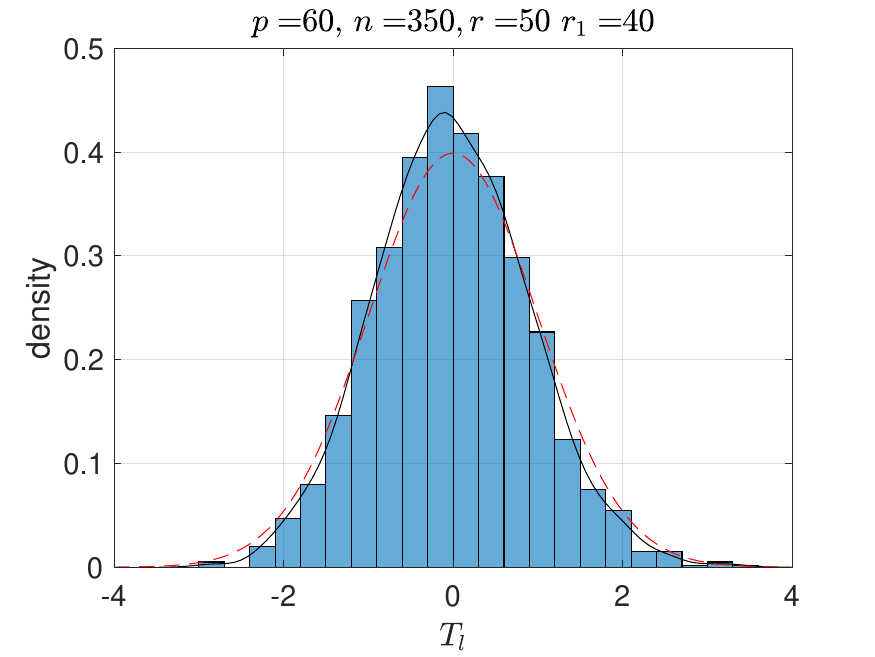}}
	\subfigure {\includegraphics[width=1.55in]{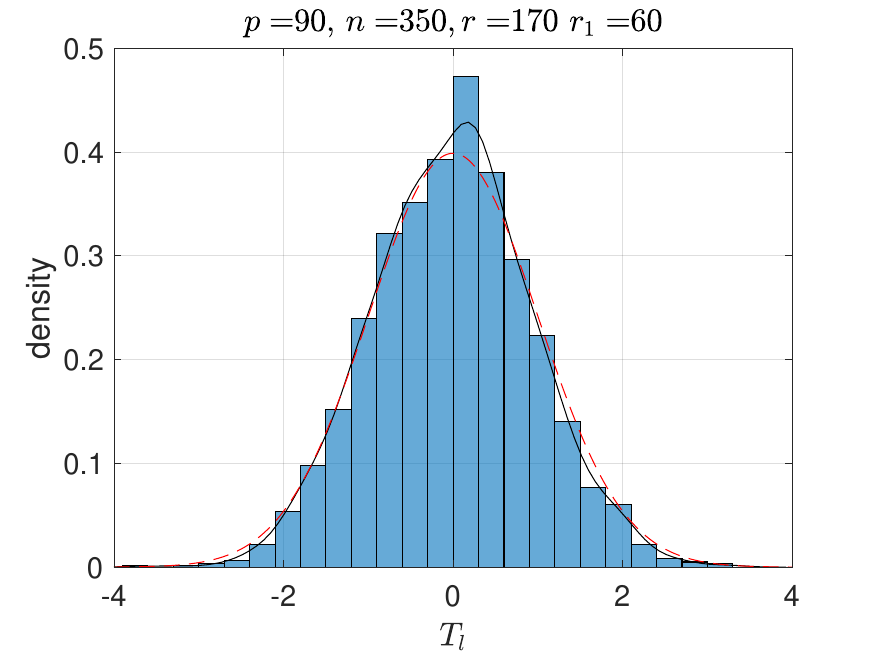}}
	\subfigure {\includegraphics[width=1.55in]{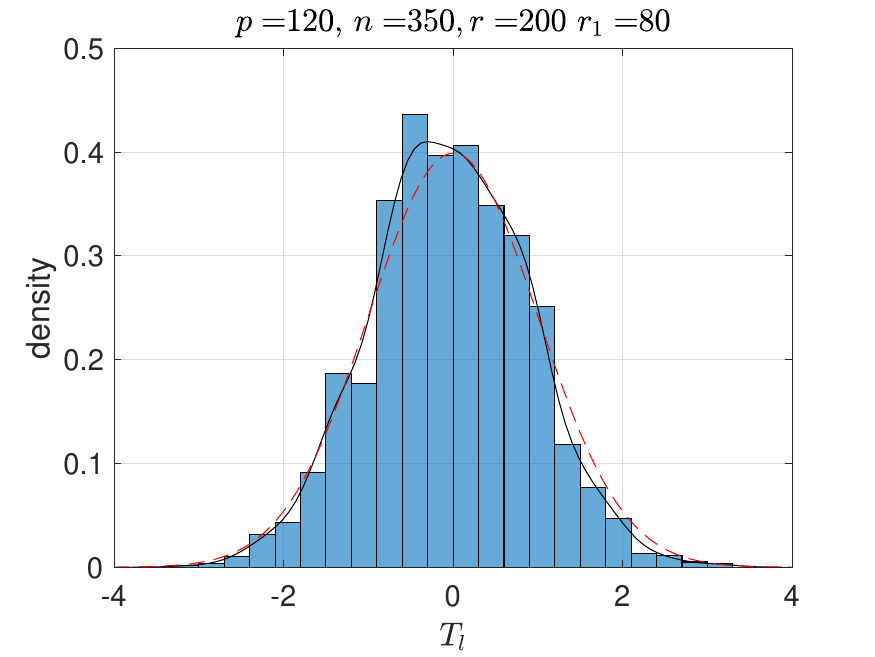}}
	\caption{Empirical distribution of $T_{l}$ under the null in Model 4. } \vspace{-1em}
	\label{fig:M4}
\end{figure}

As shown in the tables, the empirical sizes of our proposed test can  consistently achieve  correct empirical sizes around  0.05 for all models.
Moreover, as the distance between the alternative and null hypotheses increases, so does the power of the test. Figures \ref{fig:M3} and \ref{fig:M4}  depict the simulated empirical distributions
of our proposed test statistic when  the null hypothesis is true in several cases in Models 3 and 4. We can find that the empirical distribution of the statistic fits well with the standard normal distribution.
Moreover, note that the parameter $\rho$ defined in Model 4  represents the degree of correlation among the $p$ components of error vectors,  where $\rho=0$ corresponds to the case of Model 3. 
Models 3 and 4 represent scenarios where the errors are completely uncorrelated or strongly correlated. The results indicate that  our test method can provide good sizes irrespective of the correlations among the coordinates of the error vectors.
In summary,  our proposed  test method  can accurately determine the number of significant variables in different scenarios.

\subsection{Numerical studies for change point detection }\label{sec5.3}
For the change point detection problem, we conduct simulations to compare our proposed method with three other methods from the existing literature, which are referred to as M-P, MSR, and SVM, described as follows.
\begin{itemize}
	\item[(a).] \textbf{M-P:} \cite{MP2019} finds the change point by judging whether the largest eigenvalue of the sample covariance matrix falls outside the support set of the standard Mar\v{c}enko-Pastur law.
	\item[(b).] \textbf{MSR:} \cite{MSR2020}  detects the change point by judging whether the mean spectral radius (MSR) calculated from the sample data is smaller than the inner ring radius of the single ring law.
	\item[(c).] \textbf{SVM:}  With the classifications of normal data and abnormal data, \cite{SVM2015} use training data to obtain Support Vector Machines (SVM) model, and predict whether there is an abnormality in the test data.
\end{itemize}
Similar to our change point detection method,  when anomalies are detected in $s$ consecutive windows using these three methods, it is considered that a change point has been detected.

For the data generation, we consider the following two models.\\
\textbf{Model 5:} $\mathbf{x}_t \stackrel{\text { i.i.d. }}{\sim} \mathcal{N}_{p}(\boldsymbol\mu, \mathbf I_p)$ for $1 \leq t \leq\lfloor  2T/3\rfloor$, $\mathbf{x}_t \stackrel{\text { i.i.d. }}{\sim} \mathcal{N}_{p}(\boldsymbol\mu, \rho_1\mathbf I_p)$ for $\lfloor 2T/3\rfloor+1 \leq t \leq T$, where $\boldsymbol\mu=(0.6, 0.6,\ldots, 0.6) \in \mathbb{R}^p$.\\
\textbf{Model 6:} $\mathbf{x}_t=\mathbf A\mathbf F_t+\mathbf e_t$, where $\mathbf A$ is a $p\times 5$ factor loading matrix with each element  generated from a uniform
distribution $U(0.5,1.5)$, the factors $\mathbf F_t\stackrel{\text { i.i.d. }}{\sim} \mathcal{N}_{5}(\mathbf  0, \mathbf I_5)$ and the  error terms $\mathbf e_t\stackrel{\text { i.i.d. }}{\sim} \mathcal{N}_{p}(\mathbf 0, \mathbf \Sigma_e)$. Here, $\mathbf \Sigma_e=\mathbf I_p$ for  $1 \leq t \leq\lfloor  2T/3\rfloor$, and $\mathbf \Sigma_{e}=(\sigma_{ij})_{i,j=1}^p$ with $\sigma_{ij}=\rho_20.8^{|i-j|}$ for $\lfloor  2T/3\rfloor+1\leq t \leq T$. 

In both models, the generated data have a length of $T=6000$. Within these 6000 observations, two additive outliers are introduced at $t=2001$ and $t=2002$,  with  values of  $20 (\mathbf{1}_p, \mathbf{1}_p)$. Moreover,  each model  involves a parameter, $\rho_1$ or $\rho_2$, which can reflect  the magnitude of the distributional change. For instance,  in Model 5,  the larger the $\rho_1$, the larger the difference between the covariance matrices before and after the time $2T/3$. We set $\rho_1=5, 10, 20$ in Model 5, and $\rho_2=3, 6,9$ in Model 6, respectively. For each scenario, 100 pieces of simulated data are generated, with  $s=20$, and $q_{11}=q_{12}=2p$. Accuracy, defined as the percentage of correctly detected change points,  is used to evaluate the detection performance. Furthermore, note that our change point detection method requires multiple tests, increasing the risk of false positives. While the Bonferroni correction (\cite{bonferroni1936}) is commonly used to adjust for this, it is too conservative. Instead, we empirically tested six significance levels ($\alpha=0.0001, 0.0005, 0.001, 0.005, 0.01, 0.05$) to find the one that yields around 95\% accuracy. Based on absolute error comparisons, we found that $\alpha=0.0005$ minimizes error and is thus selected for further simulations. More details about the selection of $\alpha$  can be found in the supplementary material.

\begin{figure}[!t]
	\centering
	\subfigure{\includegraphics[width=5in]{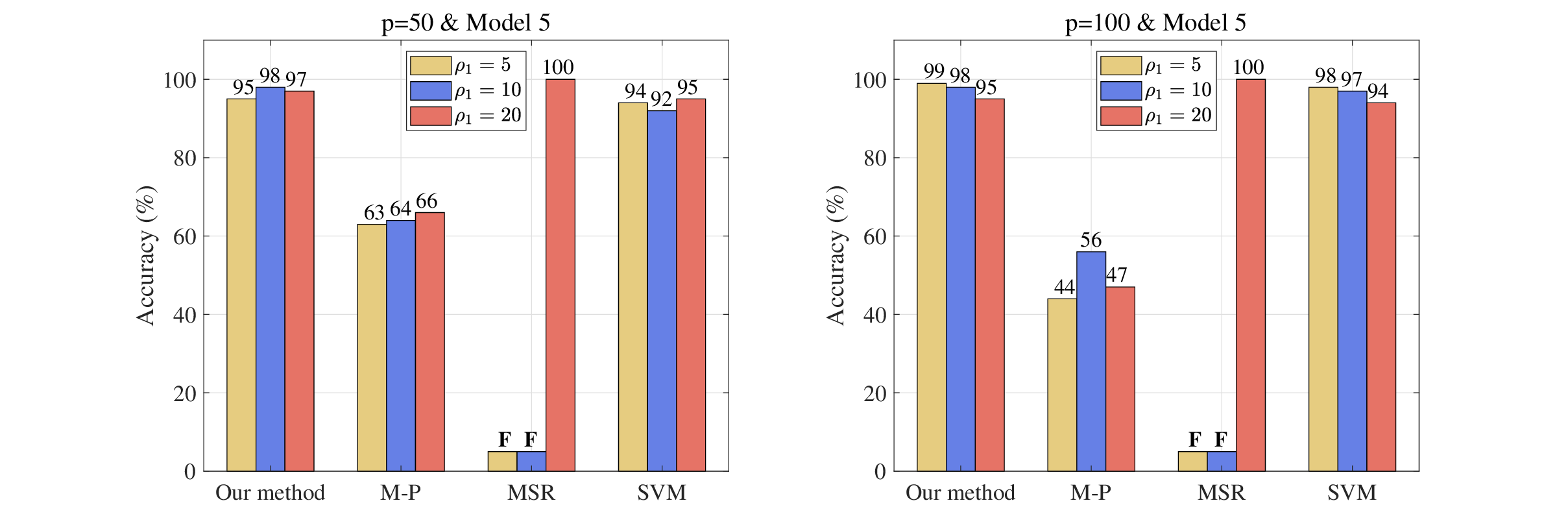}}
	\subfigure{\includegraphics[width=5in]{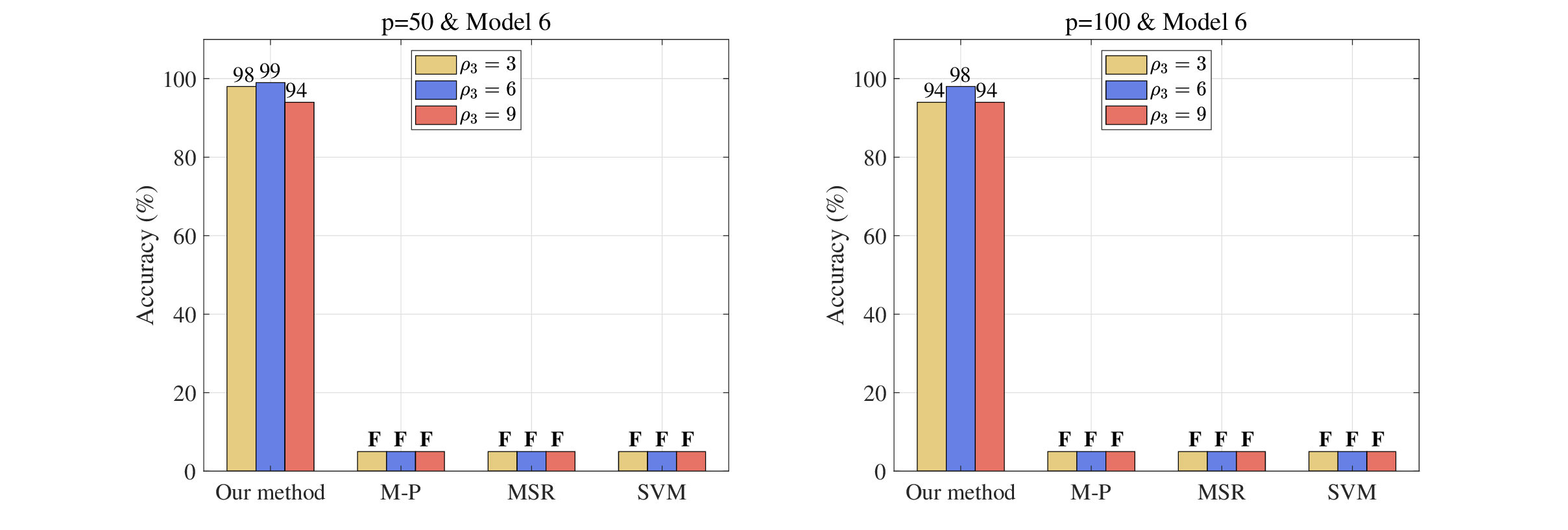}}
	\caption{Accuracy comparisons among different methods in Models 5 and 6.  The mark \rm{\textbf{F}} indicates that the method fails. }
	\label{change-point}
\end{figure}

We set the dimension $p$  to  50 and 100  for each  model and compute the accuracy values, which are reported in Figure \ref{change-point}.   
Overall, it is evident that our proposed method outperforms the other three.  Across various scenarios, the detection accuracy of our proposed test method reaches as high as about 95\%.
In contrast, the M-P method consistently demonstrates the lowest detection accuracy among both  models.
The MSR method exhibits strong performance only when the fluctuation of the change is substantial but fails when the fluctuation is minimal, such as in cases where $\rho_1=5$ and $\rho_1=10$ in Model 5. Furthermore, for the factor model setting in Model 6, both the M-P method and the MSR method fail, indicating that these methods are unsuitable for data with spiked structures. Whereas our method achieves high detection accuracy  in all cases across different models. 
Additionally, the SVM method performs  well in Model 5, achieving high accuracy. However, for Model 6 where the covariance matrix of the data exhibits a complex spiked structure, the SVM method also fails. Our method, on the other hand, continues to demonstrate high accuracy, indicating its robustness and suitability for data with complex structures.

\subsection{Empirical study}

In this section, we  evaluate the performance of our proposed change point detection method through an empirical analysis of a monthly  database for macroeconomic research (FRED-MD, \cite{fred}), which is available from  the website \url{https://research.stlouisfed.org/econ/mccracken/fred-databases/}.
The FRED-MD dataset comprises monthly series data for 128 macroeconomic variables spanning from January 1959 to June 2024. To ensure the data's suitability for analysis, we conduct several pre-processing steps: removing missing data, transforming the remaining data into a stationary form, and standardizing variables.
Following the pre-processing, the data dimension and sample size are adjusted to $p=103$ and $n=783$, respectively.

In Figure \ref{fred-MD}, we plot the changes of  the 103 features in the original data  over time. Notably, around 2008, several indicators experienced significant fluctuations, suggesting the occurrence of a change point during this period. The two most volatile indicators are the Reserves of Depository Institutions (NONBORRES) and the Monetary Base (BOGMBASE), reflecting some of the policy responses to the 2008 global financial crisis.
\begin{figure}[!t]
	\centering
	\includegraphics[width=4.5in]{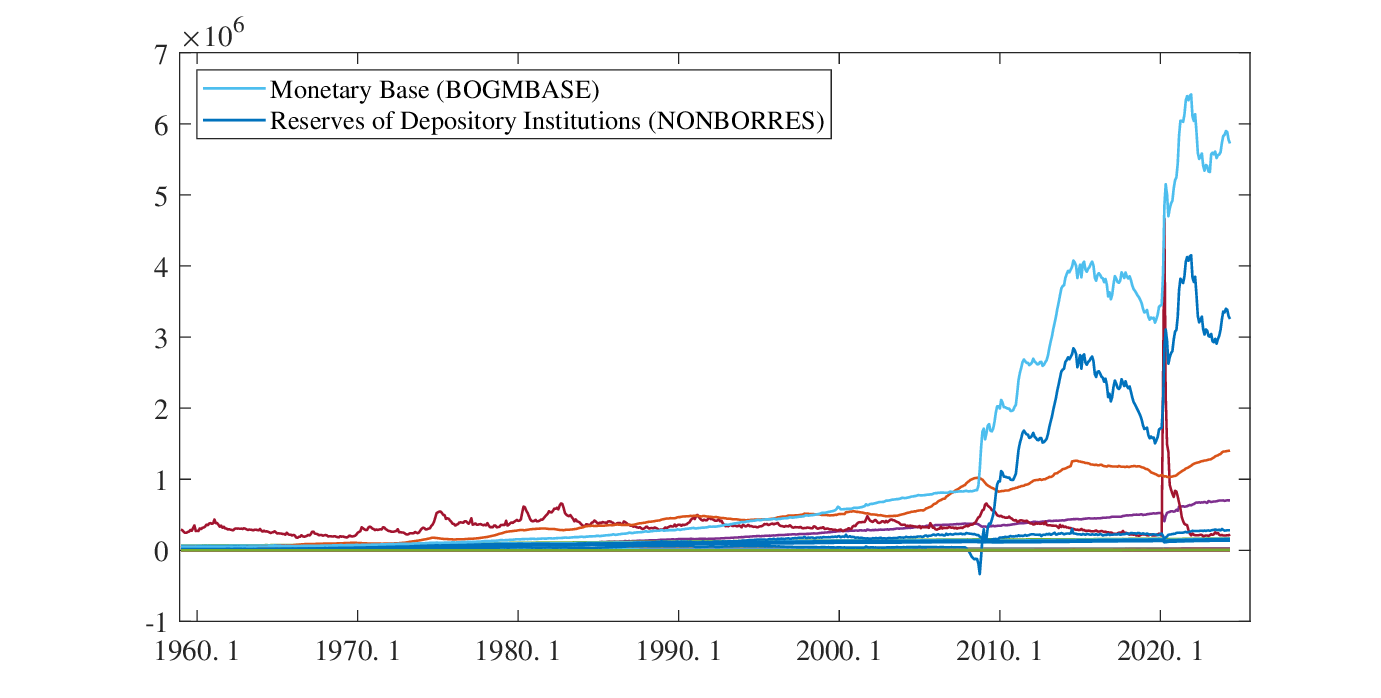}
	\caption{The monthly  database for macroeconomic research (FRED-MD). }
	\label{fred-MD}
\end{figure}

Due to the complex structure of the actual data and the challenge in satisfying the column independence condition, the asymptotic distribution of the test statistic  $T_j$ deviates to some extent from Corollary~\ref{Th3}. To address this, we refer to \cite{Chen2022} and use the upper 95\% quantile of the empirical distributions of  $T_j, j=1,2,...$ to adjust the original threshold of the rejection region.
For parameter settings, we chose $q_{11}=p+50$ and $q_{12}=p-10$, with $s$ set to 10. Change point detection was performed using Algorithm \ref{alg2}, and the results identified the change point at the 595th time point, corresponding to July 2008. This finding aligns with our analysis in Figure \ref{fred-MD}, validating the effectiveness of the proposed methods on  real-world data.

\section{Conclusion}\label{sec6}
This paper discussed the problems of testing the number of spikes in a generalized  spiked Fisher matrix.  We proposed a universal test for this hypothesis testing problem, considered the partial linear spectra statistic as the test statistic, and established its CLT under the null hypothesis. The theoretical results of our test method were applied to two problems, large-dimensional linear regression and change point detection.  The exceptional effectiveness and generality of our proposed methods were demonstrated through extensive numerical and empirical studies. Being independent of the population distribution and diagonal assumption of the population covariance matrix, our approaches  have superior performance. Furthermore, the problem of testing the spikes in a non-central spiked Fisher matrix is also important and thus warrants further research.

%\appendices
%\section{Proofs} 
%\renewcommand{\appendixname}{\Alph{subsection}}
%%\subsection{Results of Model 1}

\appendix

\section{Proof of Theorem 2.1} \label{app1}
\begin{proof}
	We need to prove the CLT of the following partial linear spectral statistic	
	\begin{equation} 
	T=\sum\limits_{j=1}^p  f\{l_j(\bF)\}-\sum\limits_{j \in  \bbJ_k, k=1}^K  f\{ l_j(\bF)\}. \label{partial}
	\end{equation}
	
	We first focus on the calculation of  the second term $\sum\limits_{j \in  \bbJ_k, k=1}^K  f\{ l_j(\bF)\}$ in \eqref{partial}.
	It is  known that, for the sample spiked  eigenvalues of $\mathbf F$, almost surely,
	\begin{align}
		\sum_{j \in \mathcal{J}_{k}, k=1}^{K} f\{ l_j(\bF)\}\rightarrow \sum_{k=1}^{K} m_k  f\left(\psi_k\right),\label{as1}
	\end{align}
	where $\psi_k$ is defined in \eqref{psik}.
	
	For the first term in \eqref{partial}, we have
	\begin{align*}
		\sum\limits_{j=1}^p  f\{l_j(\bF)\}&=p \int f(x)\mathrm{d}F_{\mathbf{n}}(x)\\
		&=\!p\! \!\int\! f(x) \mathrm{d}\left[\!F_{\mathbf{n}}(x)-F^{(c_{n1},c_{n2},H_n)}(x)\!\right] \!+\!p \int\! f(x) \mathrm{d}F^{(c_{n1},c_{n2},H_n)}(x),
	\end{align*}
	where $F_{\mathbf{n}}$ is the ESD of the Fisher matrix $\bF$, and $F^{(c_{n1},c_{n2},H_n)}$ is the distribution function in which the parameter $(c_{1},c_{2},H)$ is replaced by $(c_{n1},c_{n2}, H_n)$ in the LSD of $\bF$.
	
	We consider the following  centered and scaled variables
	$$
	p \int f(x) \mathrm{d}\left[F_{\mathbf{n}}(x)-F^{(c_{n1},c_{n2},H_n)}(x)\right]=:G_n(x).
	$$
	Therefore, we have
	$$
	G_n(x)=\sum\limits_{j=1}^p  f\{l_j(\bF)\}-p \int f(x) \mathrm{d}F^{(c_{n1},c_{n2},H_n)}(x).
	$$
	
	Under different population assumptions, both of \cite{zheng2012} and \cite{zheng2017} proved that the process $G_n(x)$
	would converge weakly to a Gaussian vector  with explicit expressions of means and covariance functions when the dimensions are proportionally large compared to the sample sizes. The arrays of two random vectors  in the Fisher matrix considered by \cite{zheng2012} can be independent but differently distributed. \cite{zheng2017} further explored  general Fisher matrices with arbitrary	population covariance matrices,  their proposed  CLT extended and covered the conclusion of \cite{zheng2012}, and can be expressed as follows
	$$
	G_n(x) \Rightarrow \mathcal{N}\left(\mu_{f,H},  \nu_{f,H}\right),
	$$
	where
	\begin{align}
		\mu_{ f,H}=&
		\frac{q}{4\pi i} \oint f(z) \md \log \bigg[
		\frac{h^2}{c_2}- \frac{c_1}{c_2} \frac{\big\{1-c_2\int \frac{m_0(z)}{t+m_0(z)} \md  H(t)\big\}^2}{1-c_2\int \frac{m^2_0(z)}{\big\{t+m_0(z)\big\}^2} \md  H(t)}
		\bigg] \non
		&\ +\frac{q}{4\pi i} \oint f(z) \md \log \bigg[
		1- c_2 \int  \frac{m^2_0(z) \md  H(t)}{\big\{t+m_0(z)\big\}^2}\bigg] \non 
		&\ -\frac{\beta_x c_1 }{2 \pi i} \oint f(z)\frac{z^2\underline{m}^3(z)h_{m_1}(z)}{\frac{h^2}{c_2}- \frac{c_1}{c_2} \frac{\big\{1-c_2\int \frac{m_0(z)}{t+m_0(z)} \md  H(t)\big\}^2}{1-c_2\int \frac{m^2_0(z)}{\big\{t+m_0(z)\big\}^2} \md  H(t)}} \md z\non
		&\ +\frac{\beta_y c_2 }{2 \pi i} \oint f(z)\um'(z)
		\frac{\underline{m}^3(z)m^3_0(z)h_{m_2}\{-\um^{-1}(z)\}}
		{1-c_2\int \frac{m^2_0(z)}{\big\{t+m_0(z)\big\}^2} \md  H(t)}
		\md z,\label{mean}\\
		\nu_{ f,H}=&-\frac{q+1}{4\pi^2}\oint\oint\frac{f(z_1)f(z_2)}{\{m_0(z_1)-m_0(z_2)\}^2}
		\md m_0(z_1)\md m_0(z_2) \non
		&-\frac{\beta_{x} c_{1}}{4 \pi^{2}} \oint \oint f\left(z_{1}\right) f\left(z_{2}\right) \frac{\partial^{2}\left[z_{1} z_{2} \underline{m}\left(z_{1}\right) \underline{m}\left(z_{2}\right) h_{v_1}\left(z_{1}, z_{2}\right)\right]}{\partial z_{1} \partial z_{2}} \md z_{1} \md z_{2} \non
		&\ -\frac{\beta_{y} c_{2}}{4 \pi^{2}} \oint \oint f\left(z_{1}\right) f\left(z_{2}\right) \non
		& \quad\cdot \frac{\partial^{2}\left[ \underline{m}\left(z_{1}\right)\! m_{0}\left(z_{1}\right) \! \underline{m}\left(z_{2}\right)\! m_{0}\left(z_{2}\right)\! h_{v_{2}}\left(- \underline{m}^{-1}\left(z_{1}\right),\!- \underline{m}^{-1}\left(z_{2}\right)\right)\right]}{\partial z_{1} \partial z_{2}} \md z_{1} \md z_{2}. \label{var}
	\end{align}	
Here, $\beta_x={\rE|x_{11}|^4-q-2}$ and $\beta_y={\rE|y_{11}|^4-q-2}$ with $q=1$ for real case and 0 for complex, 
\begin{align}
	&h_{m_1}(z)=-\frac{\frac{m_{0}^{2}(z)}{z^{2} \underline m(z)} \int \frac{t^{2}}{\left(t+m_{0}(z)\right)^{3}} \md H(t)}{1-c_{2} \int \frac{m_{0}^{2}(z)}{\left(t+m_{0}(z)\right)^{2}} \md H(t)}, \notag\\
	&h_{m_2}\left(-\underline{m}^{-1}(z)\right)=-\frac{1}{\underline{m}^{3}(z)} \int \frac{t}{\left(t+m_{0}(z)\right)^{3}} \md H(t), \notag\\
	&h_{v_1}\left(z_{1}, z_{2}\right)=\frac{1}{z_{1} z_{2} \underline{m}\left(z_{1}\right) \underline{m}\left(z_{2}\right)} \int \frac{t^{2}}{\left(t+m_{0}\left(z_{1}\right)\right)\left(t+m_{0}\left(z_{2}\right)\right)} \md H(t), \notag\\
	& h_{v_2}\left(-\underline{m}^{-1}\left(z_{1}\right),-\underline{m}^{-1}\left(z_{2}\right)\right)
	=\frac{1}{\underline{m}\left(z_{1}\right) \underline{m}\left(z_{2}\right)} \int \frac{1}{\left(t+m_{0}\left(z_{1}\right)\right)\left(t+m_{0}\left(z_{2}\right)\right)} \md H(t). \notag
\end{align}
The contours all contain the support of the LSD of  $\bF$ and 
are non-overlapping  in   \eqref{mean} and \eqref{var}.

	Then 
	$$G_n(x)-\sum\limits_{j \in  \bbJ_k, k=1}^K  f\{ l_j(\bF)\}+\sum_{k=1}^{K} m_k  f\left(\psi_k\right)\Rightarrow \mathcal{N}\left(\mu_{f,H},  \nu_{f,H}\right),$$
	which is equivalent to
		\begin{align*}
		\nu_{f,H}^{-\frac{1}{2}}\!\left\{G_n(x)-\!\!\sum\limits_{j \in  \bbJ_k, k=1}^K  f\{ l_j(\bF)\}
		\!+\!\sum_{k=1}^{K} m_k  f\left(\psi_k\right)\!- \!\mu_{f,H}\!\right\} \Rightarrow \mathcal{N} \left( 0,
		1\right).
	\end{align*}
	Proof is completed.
\end{proof}

\begin{acks}[Acknowledgments]
The authors would like to thank the anonymous referees, the Associate Editors,
and the Editor for their constructive comments  of this paper.
\end{acks}

\begin{funding}
The second author was   supported by Key technologies for coordination and interoperation of power distribution service
resource Grant No. 2021YFB2401300, and National Natural Science Foundation of China Grant No. 12326606.
\end{funding}

%%%%%%%%%%%%%%%%%%%%%%%%%%%%%%%%%%%%%%%%%%%%%%
%% supplementary Material, if any, should   %%
%% be provided in {supplement} environment  %%
%% with title and short description.        %%
%%%%%%%%%%%%%%%%%%%%%%%%%%%%%%%%%%%%%%%%%%%%%%
\begin{supplement}
Supplementary materials  contain an algorithm to approximate the center parameter term  $d_1(f,H_n) $ in Theorem \ref{Th1},  calculation details of Examples \ref{ex1} and \ref{ex2}, an illustration about  properties of the matrix $\mathbf H$ in Section \ref{sec3},  the additional simulation results for Section \ref{sec5.1}, and  details of the selection method of the significance level $\alpha$ in Section \ref{sec5.3}.
\end{supplement}

\bibliographystyle{imsart-number}
\bibliography{myref}

\end{document}